\documentclass[12pt]{article}
\usepackage{amsmath, amsthm, amsfonts, amssymb}
\usepackage[all,cmtip]{xy}
\usepackage[margin=3cm]{geometry}
\usepackage{hyperref}

\numberwithin{equation}{section}

\theoremstyle{plain}
\newtheorem{thm}{Theorem}[section]
\newtheorem{lemma}[thm]{Lemma}
\newtheorem{prop}[thm]{Proposition}
\newtheorem{cor}[thm]{Corollary}

\theoremstyle{definition}
\newtheorem{defn}{Definition}[section]
\newtheorem{exmp}{Example}[section]
\newtheorem{rmrk}[thm]{Remark}

\DeclareMathOperator{\cur}{Cur}
\DeclareMathOperator{\en}{End}
\DeclareMathOperator{\ind}{Ind}
\DeclareMathOperator{\lie}{Lie}
\DeclareMathOperator{\mat}{Mat}
\DeclareMathOperator{\perep}{PERep}
\DeclareMathOperator{\Real}{Re}
\DeclareMathOperator{\rep}{Rep}
\DeclareMathOperator{\res}{Res}

\DeclareMathOperator{\vir}{Vir}
\DeclareMathOperator{\zhu}{Zhu}

\newcommand{\C}{\mathbb{C}}
\newcommand{\R}{\mathbb{R}}
\newcommand{\Z}{\mathbb{Z}}

\newcommand{\B}{\mathcal{B}}

\newcommand{\g}{\mathfrak{g}}

\newcommand{\al}{\alpha}
\newcommand{\ga}{\gamma}
\newcommand{\G}{\Gamma}
\newcommand{\D}{\Delta}
\newcommand{\eps}{\epsilon}
\newcommand{\la}{\lambda}

\newcommand{\vac}{\left|0\right>}
\newcommand{\fp}{\left\lfloor P \right\rfloor}

\title{Higher level twisted Zhu algebras}
\date{}
\author{Jethro van Ekeren\footnote{email: jethro@math.mit.edu}\\ \small{\emph{Department of Mathematics, MIT, Cambridge, Massachusetts 02139, USA}}}

\begin{document}

\maketitle

\begin{abstract}
The study of twisted representations of graded vertex algebras is important for understanding orbifold models in conformal field theory. In this paper we consider the general set-up of a vertex algebra $V$, graded by $\G/\Z$ for some subgroup $\G$ of $\R$ containing $\Z$, and with a Hamiltonian operator $H$ having real (but not necessarily integer) eigenvalues. We construct the directed system of twisted level $p$ Zhu algebras $\zhu_{p, \G}(V)$, and we prove the following theorems: For each $p$ there is a bijection between the irreducible $\zhu_{p, \G}(V)$-modules and the irreducible $\G$-twisted positive energy $V$-modules, and $V$ is $(\G, H)$-rational if and only if all its Zhu algebras $\zhu_{p, \G}(V)$ are finite dimensional and semisimple. The main novelty is the removal of the assumption of integer eigenvalues for $H$. We provide an explicit description of the level $p$ Zhu algebras of a universal enveloping vertex algebra, in particular of the Virasoro vertex algebra $\vir^c$ and the universal affine Kac-Moody vertex algebra $V^k(\g)$ at non-critical level. We also compute the inverse limits of these directed systems of algebras.
\end{abstract}

\section{Introduction} \label{intro}

To a vertex operator algebra $V$, Zhu \cite{Zhu} attached a unital associative algebra which he called $A(V)$ and which we call $\zhu(V)$ -- the Zhu algebra of $V$. The Zhu algebra is important because much of the representation theory of $V$ can be reduced to representation theory of $\zhu(V)$. Indeed if $M = \oplus_{j \in \Z_+} M_j$ is a \emph{positive energy $V$-module} (defined in section \ref{prelim}) then $M_0$ naturally acquires the structure of a $\zhu(V)$-module, and $M \mapsto M_0$ defines a restriction functor $\Omega$ from the category $\perep(V)$ of positive energy $V$-modules to the category $\rep(\zhu(V))$ of $\zhu(V)$-modules.

There is an induction functor $L$, going in the opposite direction, such that $\Omega \circ L = \text{id}_{\rep(\zhu(V))}$. Furthermore $\Omega$ and $L$ become mutually inverse equivalences upon restriction from $\perep(V)$ to the full subcategory of \emph{almost irreducible} positive energy $V$-modules \cite{DK} (the term `almost irreducible' will be defined in section \ref{funcprop}).

In \cite{DLM} Dong, Li, and Mason introduced, for each $p \in \Z_+$, the level $p$ Zhu algebra $\zhu_p(V)$. It is a unital associative algebra with the property that for $M \in \perep(V)$, $M_p$ is naturally a $\zhu_p(V)$-module. There is a restriction functor $\Omega_p : M \mapsto M_p$, and a corresponding induction functor $L^p$, as before. Zhu's original algebra is the special case $\zhu(V) = \zhu_0(V)$.

Meanwhile, in \cite{DLM2}, the same authors constructed the \emph{$g$-twisted} Zhu algebra $\zhu_g(V)$ of a vertex operator algebra $V$ carrying a finite-order automorphism $g$. There are functors going between the categories of $g$-twisted positive energy $V$-modules and $\zhu_{g}(V)$-modules. Their definition reduces to Zhu's in the case $g = 1$.

In \cite{DK}, De Sole and Kac develop the same theory in the more general setting of graded vertex algebras with energy operator $H$ having not necessarily integer eigenvalues, thus encompassing some important examples such as the affine $W$-algebras.

In \cite{DLM3} twisted level $p$ Zhu algebras are defined, as in \cite{Zhu}, \cite{DLM} and \cite{DLM2}, with the assumption of integral eigenvalues of $H = L_0$. In the present paper we remove this restriction, thus subsuming all these generalizations. Our approach closely follows \cite{DK}, incorporating much from \cite{DLM}.

We compute the level $p$ Zhu algebra of a universal enveloping vertex algebra of a Lie conformal algebra with a Virasoro element, this is done in the untwisted case for simplicity. Examples of such vertex algebras are the universal affine Kac-Moody vertex algebra $V^k(\g)$ at non-critical level $k \neq -h^\vee$ and the Virasoro vertex algebra $\vir^c$.

I would like to thank my Ph.D. advisor Victor Kac for many useful suggestions and discussions.

\subsection{Outline of the paper}

In section \ref{prelim} we recall the definitions of vertex algebras and positive energy modules of a vertex algebra in the ordinary and twisted cases, these definitions are the same as those in \cite{DK}. In section \ref{motivation} we give some motivation for introducing $\zhu_p(V)$ before presenting, in section \ref{defdef}, the precise definition of the $\G$-twisted level $p$ Zhu algebra $\zhu_{p, \G}(V)$ of a $\G/\Z$-graded vertex algebra $V$.

In subsequent sections we prove unitality and associativity of $\zhu_{p, \G}(V)$. Proof by direct calculation is lengthy, but for us these calculations are made bearable by the introduction, following \cite{DK}, of a modified state-field correspondence
\begin{align*}
Z(a, w) = (1 + w)^{p+\D_a} Y(a, w).
\end{align*}
We prove analogs of the Borcherds identity, and the skew-symmetry identity for $Z(a, w)$.

In section \ref{maps} we show that there are surjective algebra homomorphisms $\zhu_{p+1}(V) \twoheadrightarrow \zhu_p(V)$ for each $p \in \Z_+$. The question ``Is $\zhu_{p, \G}(V) \neq 0$?'' thus reduces to the same question for $p=0$. This is a hard question (see Remark \ref{twistedexistence} and section \ref{unisec}); since there is no general proof that $\zhu_{p, \G}(V) \neq 0$ for all choices of $p$, $V$ and $\G$, whenever we state that $\zhu_{p, \G}(V)$ is unital, we mean unital or zero.

In section \ref{repnthry} we define the restriction functor $\Omega_p$, and the induction functor $L^p$, relating positive energy twisted $V$-modules to $\zhu_{p, \G}(V)$-modules. We show that $\Omega_p$ and $L^p$ become mutually inverse equivalences between $\rep(\zhu_{p, \G}(V))$ and the full subcategory of $\perep(V)$ of \emph{$p$-irreducible} positive energy $V$-modules. We also show that $V$ is \emph{$(\G, H)$-rational} if and only if all the Zhu algebras $\zhu_{p, \G}(V)$ are finite dimensional semisimple algebras (unfamiliar terms will be defined in section \ref{repnthry}).

In section \ref{examples} we compute the level $p$ Zhu algebra of a universal enveloping vertex algebra by using the universal properties of the objects in question. Universal enveloping vertex algebras are part of the theory of Lie conformal algebras (see \cite{Kac} and \cite{DK}), we give definitions and basic theorems in the section for completeness. We compute the inverse limit of the directed system of Zhu algebras of a universal enveloping vertex algebra. Finally we make some general remarks about Zhu algebras and their inverse limits for rational vertex algebras.

We have taken $p \in \Z_+$ throughout this paper for simplicity, but the same theorems can be formulated and proved for arbitrary $p \in \R_+$. In the untwisted case this gives nothing new, but in the twisted case non-integer level Zhu algebras may arise. In appendix \ref{pnotinZ} we detail the necessary changes to the relevant definitions and proofs.


\subsection{Basic definitions} \label{prelim}

The calculus of formal power series is a useful tool for making calculations in vertex algebras. Here we briefly define the notation we use, and refer the reader to \cite{Kac} for details. The formal delta function $\delta(z, w) \in \C[[z, z^{-1}, w, w^{-1}]]$ is defined by
\[
\delta(z, w) = \sum_{n \in \Z} z^n w^{-n-1}.
\]
If $U$ is a vector superspace\footnote{In this paper all objects are assumed to be super-, and linear maps parity-preserving, unless otherwise stated. Thus `algebra' implicitly means `superalgebra', etc. The reader may ignore this detail if he wishes.}, the residue operation $\res_z : U[[z, z^{-1}]] \rightarrow U$ sends $f(z) = \sum f_{(n)} z^{-n-1}$ to its $z^{-1}$ coefficient $f_{(0)}$. The operators $i_{z, w}$ and $i_{w, z}$ denote expansion as Laurent series in the domains $|z| > |w|$ and $|w| > |z|$, respectively. For example
\begin{align*}
i_{z, w} (z-w)^{-1} = \sum_{j \in \Z_+} z^{-j-1} w^j
\quad \text{and} \quad
i_{w, z} (z-w)^{-1} = -\sum_{j \in \Z_+} z^j w^{-j-1}.
\end{align*}

Throughout the paper we use the notation $x^{(n)} = \frac{x^n}{n!}$ and $\partial_z^{(n)} f(z) = \frac{1}{n!} \frac{\partial^n f}{{\partial z}^n}$. We write $[\xi^n] : f(\xi)$ for the $\xi^n$ coefficient of a formal power series $f(\xi)$.

Let $U$ and $V$ be vector spaces. A linear map $V \rightarrow (\en U)[[z, z^{-1}]]$, $a \mapsto Y(a, z)$, is called a quantum field if $Y(a, z)b \in U((z))$ for all $a \in V$, $b \in U$.
\begin{defn}
A vertex algebra $(V, \vac, Y)$ is a vector superspace $V$ together with a nonzero even element $\vac \in V$ and an injective parity-preserving linear map,
\begin{align*}
Y : V \otimes V &\rightarrow V((z)) \\
a \otimes b &\mapsto Y(a, z)b = \sum_{n \in \Z} (a_{(n)}b) z^{-n-1}
\end{align*}
(i.e., $Y(a, z)$ is a quantum field for all $a \in V$, equivalently $a_{(n)}b = 0$ for $n \gg 0$), such that the following two axioms hold:
\begin{itemize}
\item The vacuum axiom, $Y(\vac, z) = I_V$.

\item The Borcherds identity,
\begin{align*}
& \sum_{j \in \Z_+} Y(a_{(n+j)}b, w) \partial_w^{(j)} \delta(z, w) \\
& \phantom{\lim} = Y(a, z) Y(b, w) i_{z, w}(z-w)^n - p(a, b) Y(b, w) Y(a, z) i_{w, z}(z-w)^n
\end{align*}
for all $a, b \in V$, $n \in \Z$.
\end{itemize}
\end{defn}
Here, and further, $p(a, b)$ stands for $(-1)^{p(a)p(b)}$ where $p(a)$ is the parity of $a \in V$.

Let $U \subseteq V$ be a subspace of the vertex algebra $V$. Then $U$ is a vertex subalgebra of $V$ if $\vac \in U$ and $a_{(n)}b \in U$ for all $a, b \in U$, $n \in \Z$. $U$ is an ideal if $a_{(n)}b, b_{(n)}a \in U$ for all $a \in V$, $b \in U$, $n \in \Z$. A simple vertex algebra is one with no proper nonzero ideals. A linear map $f : V_1 \rightarrow V_2$ is a homomorphism of vertex algebras if $f(\vac_1) = \vac_2$ and $f(a_{(n)}b) = f(a)_{(n)}f(b)$ for all $a, b \in V_1$, $n \in \Z$.

For any vertex algebra $V$ the \emph{translation operator} is defined by
\begin{align*}
T : V &\rightarrow V \\
a &\mapsto a_{(-2)} \vac.
\end{align*}
One may check that $Y(Ta, z) = \partial_z Y(a, z)$ by putting $n = -2$, $b = \vac$ in the Borcherds identity and taking the residue $\res_z$ of both sides.

We consider vertex algebras equipped with the extra structure of an \emph{energy operator} $H$. This is a diagonalizable operator $H : V \rightarrow V$ such that
\begin{align}
[H, Y(a, z)] = z \partial_z Y(a, z) + Y(Ha, z) \label{Hdef}
\end{align}
for all $a \in V$. We assume $H$ has real eigenvalues on $V$ for simplicity.

A \emph{Virasoro element} $\omega \in V$ is an even element, such that if
\[
Y(\omega, z) = L(z) = \sum_{n \in \Z} L_n z^{-n-2}
\]
then $L_{-1} = T$, $L_0$ is diagonalizable, and $\{L_n\}_{n \in \Z}$ has the commutation relations of the Virasoro Lie algebra,
\begin{align}
[L_m, L_n] = (m-n) L_{m+n} + \delta_{m, -n} \frac{m^3-m}{12} c I_V
\end{align}
for some fixed $c \in \C$ called the central charge of $(V, \vac, Y, \omega)$; then $L_0$ is an energy operator. A vertex algebra with a fixed Virasoro element is called a \emph{vertex operator algebra}. To define Zhu algebras we do not require a Virasoro element, only an energy operator, but to prove some theorems we will need to assume that $H = L_0$ for some Virasoro element $\omega$. The affine Kac-Moody vertex algebra $V^k(\g)$ at the critical level $k = -h^\vee$ is an example of a vertex algebra with an energy operator that is not a vertex operator algebra. For all other $k$ though, $V^k(\g)$ is a vertex operator algebra.

If $a$ is an eigenvector of $H$ with eigenvalue $\D_a$ then $\D_a$ is called the \emph{conformal weight} of $a$. It shall be understood that, whenever we refer to $\D_a$ in a definition, the definition holds as stated for $a$ of homogeneous conformal weight, and is extended linearly to all $a \in V$. The conformal weight has the following three properties:
\begin{align}
\D_{\vac} = 0, \quad
\D_{Ta} = \D_a + 1, \quad \text{and} \quad
\D_{a_{(n)}b} = \D_a + \D_b - n - 1. \label{deltaprops}
\end{align}
The first of these follows upon substituting $a = \vac$ in (\ref{Hdef}) and using injectivity of the map $Y$, the third is proved by applying (\ref{Hdef}) to $b$ and examining the $z^{-n-1}$ coefficient, the second is the special case of the third where $n=-2$ and $b = \vac$.

It is convenient to introduce a second indexing of the operators $a_{(n)}$, called the conformal weight indexing, defined by
\begin{align*}
a(z) = \sum_{n \in \Z} a_{(n)}z^{-n-1} = \sum_{n \in -[\D_a]} a_n z^{-n-\D_a}.
\end{align*}
Here, and further, $[\alpha]$ denotes the coset $\alpha + \Z$ of $\alpha \in \R$ modulo $\Z$. Equivalently $a_n = a_{(n+\D_a-1)}$. From (\ref{Hdef}) we find
\begin{align} \label{Hcomm}
[H, a_{(n)}] = (\D_a-n-1) a_{(n)} \quad \text{and} \quad [H, a_n] = -n a_n.
\end{align}

In preparation for the definition of a $V$-module it is useful rewrite the Borcherds identity by expanding and equating coefficients; it becomes
\begin{align}
\begin{split}
& \sum_{j \in \Z_+} \binom{m'}{j} (a_{(n+j)}b)_{(m'+k'-j)}c \\
& \phantom{\lim} = \sum_{j \in \Z_+} (-1)^j \binom{n}{j} \left[ a_{(m'+n-j)} b_{(k'+j)} - p(a, b) (-1)^n b_{(k'+n-j)} a_{(m'+j)} \right]c
\end{split} \label{norbor}
\end{align}
for all $a, b, c \in V$, $m', k', n \in \Z$. Substituting $m = m' - \D_a + 1$ and $k = k' + n - \D_b + 1$ in equation (\ref{norbor}) yields
\begin{align}
\begin{split}
& \sum_{j \in \Z_+} \binom{m+\D_a-1}{j} (a_{(n+j)}b)_{m+k}c \\
& \phantom{\lim} = \sum_{j \in \Z_+} (-1)^j\binom{n}{j} \left[ a_{m+n-j}b_{k+j-n} - p(a, b) (-1)^n b_{k-j} a_{m+j} \right]c
\end{split} \label{conbor}
\end{align}
for all $a, b, c \in V$, $m \in -[\D_a]$, $k \in -[\D_b]$ and $n \in \Z$.

\begin{defn} \label{ordmoddef}
Let $(V, \vac, Y)$ be a vertex algebra. A \emph{$V$-module} is a vector superspace $M$, together with a linear parity-preserving map
\begin{align*}
Y^M : V \otimes M &\rightarrow M((z)) \\
a \otimes x &\mapsto Y^M(a, z)x = \sum_{n \in \Z} (a^M_{(n)}x) z^{-n-1},
\end{align*}
such that the following two properties hold:
\begin{itemize}
\item $Y^M(\vac, z) = I_M$.

\item For all $a, b \in V$, $x \in M$, $m', k', n \in \Z$,
\begin{align*}
\begin{split}
& \sum_{j \in \Z_+} \binom{m'}{j} (a_{(n+j)}b)^M_{(m'+k'-j)} x \\
& \phantom{\lim} = \sum_{j \in \Z_+} (-1)^j \binom{n}{j} \left[ a^M_{(m'+n-j)} b^M_{(k'+j)} - p(a, b) (-1)^n b^M_{(k'+n-j)} a^M_{(m'+j)} \right] x.
\end{split}\end{align*}
\end{itemize}
A linear subspace $N \subseteq M$ is a $V$-submodule if $a^M_{(n)}x \in N$ for all $a \in V$, $x \in N$, $n \in \Z$. A homomorphism of $V$-modules is a linear parity-preserving map $f : M_1 \rightarrow M_2$ such that $f(a^{M_1}_{(n)}x) = a^{M_2}_{(n)}f(x)$ for all $a \in V$, $x \in M_1$, $n \in \Z$.
\end{defn}

\begin{defn} \label{confmoddef}
Let $V$ be a vertex algebra with energy operator $H$, and let $M$ be a $V$-module. For $n \in -[\D_a]$ let $a^M_n = a^M_{(n+\D_a-1)}$. A \emph{positive energy $V$-module} is a $\Z_+$-graded $V$-module
\begin{align}
M = \bigoplus_{j \in \Z_+} M_j
\end{align}
where $a^M_n (M_j) \subseteq M_{j-n}$. We call $M$ irreducible if it contains no proper nonzero graded $V$-submodules.
\end{defn}


Now we turn to the `twisted' theory.
\begin{defn}
Let $\G$ be a subgroup of $(\R, +)$ containing $\Z$. A \emph{$\G/\Z$-graded vertex algebra} is a vertex algebra $(V, \vac, Y)$ where $V$ is a $\G/\Z$-graded vector superspace, $\vac \in V^{[0]}$, and if $[\alpha], [\beta] \in \G/\Z$ then $(V^{[\alpha]})_{(n)}(V^{[\beta]}) \subseteq V^{[\alpha] + [\beta]}$ for all $n \in \Z$. An energy operator for $V$ is an energy operator $H$ for $V$ as an ordinary vertex algebra, such that $H(V^{[\alpha]}) \subseteq V^{[\alpha]}$. An element of $V^{[\alpha]}$ is said to have \emph{degree} $[\alpha]$ and we let $[\ga_a]$ denote the degree of a homogeneous element $a \in V$.
\end{defn}

\begin{exmp}
If $V$ is a vertex algebra with an automorphism $g$ such that $g^n = 1$ and $g$ commutes with $H$ then $V$ decomposes into eigenspaces $V^{[\alpha]}$ with $g$-eigenvalue $e^{2\pi i \alpha}$ (where $\alpha \in \frac{1}{n}\Z$). This makes $V$ into a $\G/\Z$-graded vertex algebra where $\G = \frac{1}{n}\Z$.
\end{exmp}

\begin{exmp} \label{ga=D}
Another important example is given by setting $[\ga_a] = [\D_a]$ for all $a$ homogeneous of conformal weight $\D_a$. Because of (\ref{deltaprops}) $V$ becomes $\G/\Z$-graded where $\G \subseteq \R$ is the subgroup generated by $1$ and the eigenvalues of $H$. 
\end{exmp}

In the following definition $M\{\{z\}\}$ denotes $\oplus_{r \in \R} z^r M((z))$.
\begin{defn}
Let $V$ be a $\G/\Z$-graded vertex algebra. A \emph{$\G$-twisted $V$-module} is a vector superspace $M$, together with a linear parity-preserving map
\begin{align*}
Y^M : V \otimes M &\rightarrow M\{\{z\}\} \\
a \otimes x &\mapsto a^M(z) = \sum_{n \in [\ga_a]} (a^M_{(n)}x) z^{-n-1}
\end{align*}
for $a$ homogeneous of degree $[\ga_a]$, such that the following two properties hold:
\begin{itemize}
\item $Y^M(\vac, z) = I_M$.

\item For all $a, b \in V$, $x \in M$, $m' \in [\ga_a]$, $k' \in [\ga_b]$ and $n \in \Z$,
\begin{align*}
\begin{split}
& \sum_{j \in \Z_+} \binom{m'}{j} (a_{(n+j)}b)^M_{(m'+k'-j)} x \\
& \phantom{\lim} = \sum_{j \in \Z_+} (-1)^j \binom{n}{j} \left[ a^M_{(m'+n-j)} b^M_{(k'+j)} - p(a, b) (-1)^n b^M_{(k'+n-j)} a^M_{(m'+j)} \right] x.
\end{split}\end{align*}
\end{itemize}
\end{defn}

\begin{rmrk} \label{twistednotthesame}
An ordinary $V$-module for a vertex algebra $V$ that happens to be $\G/\Z$-graded need not be a $\G$-twisted $V$-module. Indeed while $V$ is itself a $V$-module in the obvious way, a $\G/\Z$-graded vertex algebra $V$ is not necessarily a $\G$-twisted $V$-module (the Borcherds identity with $m, k, n \in \Z$ holds, but the Borcherds identity with non-integer $m, k$ need not hold).
\end{rmrk}

\begin{defn} \label{twistedconfmoddef}
Let $V$ be a $\G/\Z$-graded vertex algebra with energy operator $H$ and let $M$ be a $\G$-twisted $V$-module. For $n \in [\ga_a]-[\D_a]$ let $a^M_n = a^M_{(n+\D_a-1)}$, and put $[\eps_a] = [\ga_a]-[\D_a]$. A \emph{$(\G, H)$-twisted positive energy $V$-module} is an $\R_+$-graded $\G$-twisted $V$-module
\begin{align}
M = \bigoplus_{j \in \R_+} M_j
\end{align}
where $a^M_n (M_j) \subseteq M_{j-n}$. Let $\eps_a \in \R$ to be the largest non-positive element of $[\eps_a]$, so $-1 < \eps_a \leq 0$, and let $\ga_a = \D_a + \eps_a$, i.e., the largest element of $[\ga_a]$ that does not exceed $\D_a$.
\end{defn}

Why do we write $(\G, H)$-twisted instead of $\G$-twisted in the definition above? Let $V^{[\alpha]}[\D]$ denote the set of elements of $V$ homogeneous of degree $[\alpha]$ and homogeneous of conformal weight $\D$, and let $\overline{\G}$ be the subgroup of $\R$ generated by $1$ and $\eps_a$ as $a$ ranges over all nonzero $V^{[\alpha]}[\D]$. For instance, in example (\ref{ga=D}), $[\eps_a] = [\ga_a] - [\D_a] = 0$ for all $a \in V$ and so $\overline{\G} = \Z$, even though $\G/\Z$ may be nontrivial. On the other hand, if we take the trivial grading $\G = \Z$ then $\overline{\G}$ is the subgroup of $\R$ generated by $1$ and the eigenvalues of $H$. In general $\overline{\G}$ may be quite unrelated to $\G$.

The action of $V$ on $M$ preserves cosets of $\overline{\G} \subseteq \R$, so graded pieces of $M$ in different cosets of $\overline{\G}$ lie in different direct summands of $M$. Therefore we may restrict our attention to modules that are $(\overline{\G} \cap \R_+)$-graded without loss of generality. Since the index set for the grading depends on $\G$ and $H$ it is natural to call such modules $(\G, H)$-twisted.

Just as we pass from equation (\ref{norbor}) to (\ref{conbor}), we have
\begin{align} \label{modconbor}
\begin{split}
& \sum_{j \in \Z_+} \binom{m+\D_a-1}{j} (a_{(n+j)}b)^M_{m+k} x \\
& \phantom{\lim} = \sum_{j \in \Z_+} (-1)^j\binom{n}{j} \left[ a^M_{m+n-j} b^M_{k+j-n} - p(a, b) (-1)^n b^M_{k-j} a^M_{m+j} \right] x
\end{split}
\end{align}
for all $a, b \in V$ a $\G/\Z$-graded vertex algebra, $x \in M$ a $\G$-twisted $V$module, $m \in [\eps_a]$, $k \in [\eps_b]$, and $n \in \Z$.

\begin{rmrk} \label{twistedexistence}
From its definition below it is clear that $\zhu_{p, \G}(V) = 0$ if and only if $V$ has no $(\G, H)$-twisted modules. It is not known whether every $\G/\Z$-graded vertex algebra has a non-zero $\G$-twisted $V$-module or not \cite{DLM2}, but we know of no examples which don't and there are many examples which do. A general result in this direction is in \cite{DLMorbifold} where $\zhu_{0, \G}(V) \neq 0$ is proved for $V$ a simple vertex algebra with a finite automorphism $g$, and satisfying a certain finiteness condition (namely the $C_2$ condition of Zhu \cite{Zhu}).
\end{rmrk}

\begin{lemma} \label{THann}
Let $V$ be a $\G/\Z$-graded vertex algebra, let $M$ be a $(\G, H)$-twisted $V$-module, let $a \in V$, and let $s \in [\eps_a]$. Then
\begin{align} \label{THcondition}
[(T + s + \D_a)a]_s M = 0.
\end{align}
In particular $[(T+H)a]_0 M = 0$ for all $a$ such that $\eps_a = 0$.
\end{lemma}

\begin{proof}
In (\ref{conbor}) put $b = \vac$, $m = s+1$, $k = -1$, and $n = -2$, then note that $\vac^M_{(n)} = \delta_{n, -1}I_M$. We obtain
\begin{align*}
\left[ a_{(-2)}\vac + (s+\D_a) a_{(-1)}\vac \right]_{s} = 0,
\end{align*}
which is (\ref{THcondition}).
\end{proof}

Let $V$ be a vertex algebra with trivial grading. There is an important Lie superalgebra associated to $V$ called $\lie V$ \cite{Kac}, it is defined as
\begin{align}
\lie V = V[t, t^{-1}] / (T + \partial_t)V[t, t^{-1}]
\end{align}
with the Lie superalgebra bracket
\begin{align}
[a_{(m)}, b_{(k)}] = \sum_{j \in \Z_+} \binom{m}{j} (a_{(j)}b)_{(m+k-j)}, \label{norbrak}
\end{align}
where $a_{(m)}$ denotes the image of $at^m$ in the quotient.

Equation (\ref{norbrak}) is the same as equation (\ref{norbor}) with $n=0$. Let $M$ be a $V$-module. Because of (\ref{THcondition}) we find $(T + \partial_t)V[t, t^{-1}]$ acts by $0$ on $M$. These two facts imply that $M$ has the structure of a $\lie V$-module under $a_{(n)} \mapsto a^M_{(n)}$.

If $V$ has an energy operator then we define the conformal weight indexing of elements of $\lie V$ by $a_n = a_{(n+\D_a-1)}$. Equation (\ref{norbrak}) becomes
\begin{align}
[a_m, b_k] = \sum_{j \in \Z_+} \binom{m+\D_a-1}{j} (a_{(j)}b)_{m+k}. \label{conbrak}
\end{align}
\begin{defn} \label{lieVdef}
Let $V$ be a $\G/\Z$-graded vertex algebra, and let
\[
Q = \bigoplus_{[\al] \in \G/\Z} t^{\al} V^{[\al]} \otimes_\C \C[t, t^{-1}]
\]
(where $\al$ is any representative of the coset $[\al]$). Define $\lie V$ to be the quotient
\[
Q / (T + \partial_t) Q,
\]
with the Lie bracket (\ref{norbrak}), where $a_{(m)}$ is again the image of $at^m$ in the quotient, though now $m \in [\ga_a]$.

The conformal weight grading is again defined by $a_n = a_{(n+\D_a-1)}$ where $n \in [\eps_a]$ (one may check that $(T + \partial_t)Q$ is the subspace spanned by elements of the form $(Ta)_n + (n + \D_a)a_n$ where $a \in V$, $n \in [\eps_a]$, so the conformal weight grading is well-defined). Equation (\ref{conbrak}) holds. Moreover, with the grading $\deg a_n = n$, $\lie V$ is a graded Lie superalgebra.
\end{defn}
Every $(\G, H)$-twisted positive energy $V$-module is automatically a $\lie V$-module. Observe that the weight $0$ subspace $(\lie{V})_0 = \C\{a_n \in \lie V | [\eps_a] = [0], n \in \Z\} \subseteq \lie{V}$ is a Lie subalgebra.

Finally, we record some properties of the functions $\eps_a$ and $\ga_a$ defined in definition (\ref{twistedconfmoddef}). $T$ preserves $[\ga_a]$ and $[\D_a]$, hence $[\eps_a]$ and $\eps_a$. The $[\eps]$ function is additive in the $n^{\text{th}}$ products, but $\eps$ is not, for example if $-1 < \eps_a, \eps_b \leq -\frac{1}{2}$ then $\eps_{a_{(n)}b} = \eps_a + \eps_b + 1$.
\begin{defn}
\begin{align*}
\chi(a, b) = \left\{ \begin{array}{ll}
1 & \text{if $\eps_a + \eps_b \leq -1$} \\
0 & \text{if $\eps_a + \eps_b > -1$} \\
\end{array} \right..
\end{align*}
\end{defn}
We have
\begin{lemma}
\begin{align*}
\eps_{\vac} = 0, \quad
\eps_{Ta} = \eps_a, \quad \text{and} \quad
\eps_{a_{(n)}b} = \eps_a + \eps_b + \chi(a, b),
\end{align*}
and
\begin{align*}
\ga_{\vac} = 0, \quad
\ga_{Ta} = \ga_a + 1, \quad \text{and} \quad
\ga_{a_{(n)}b} = \ga_a + \ga_b - n - 1 + \chi(a, b).
\end{align*}
\end{lemma}


\subsection{Motivation to introduce Zhu algebras} \label{motivation}

Let $V$ be a $\G/\Z$-graded vertex algebra with energy operator $H$, and let $M$ be a $(\G, H)$-twisted positive energy $V$-module. Define
\begin{align}
V_\G = \{a \in V | [\eps_a] = [0]\},
\end{align}
which is a vertex subalgebra of $V$. If $a \in V_\G$ then each graded piece $M_p$ of $M$ is stable under $a^M_0$. This may lead one to try to regard $M_p$ as a module over $V_\G$ with $a \in V_\G$ acting as $a^M_0$. More precisely, one tries to find a product $*_p : V_\G \otimes V_\G \rightarrow V_\G$ such that $(a *_p b)^M_0 = a^M_0 b^M_0$ for every $V$-module $M$.

Let $x \in M_p$ where $p \in \Z_+$ for simplicity (in appendix \ref{pnotinZ} we will describe the changes that must be made for the general case of $p \in \R_+$). Let $a, b \in V$ with $[\eps_a] + [\eps_b] = [0]$, so $\eps_a + \eps_b + \chi(a, b) = 0$. Put $m = p+1+\eps_a$, $k = -(p+1+\eps_a)$, $n \in \Z$ in (\ref{modconbor}). Because $a_s M_p = 0$ for $s > p$, we obtain
\begin{align}
\sum_{j \in \Z_+} \binom{\ga_a+p}{j}(a_{(n+j)}b)^M_0 x
= \sum_{j \in \Z_+} (-1)^j \binom{n}{j} a^M_{p+1+n-j+\eps_a} b^M_{-p-1-n+j-\eps_a} x. \label{biggerp}
\end{align}
The right hand side vanishes when $n \leq -2p-2+\chi(a, b)$, in other words
\begin{align}
(a_{[n]}b)^M_0 M_p = 0 \quad \text{whenever} \quad n \leq -2p-2+\chi(a, b), \label{minus2ann}
\end{align}
where we introduce the notation
\begin{align*}
a_{[n]}b = \sum_{j \in \Z_+} \binom{\ga_a+p}{j} a_{(n+j)}b.
\end{align*}

Let $a, b \in V_\G$, so that $\eps_a = \eps_b = \chi(a, b) = 0$; equation (\ref{biggerp}) becomes
\begin{align*}
(a_{[n]}b)^M_0 x = \sum_{j \in \Z_+} (-1)^j \binom{n}{j} a^M_{p+1+n-j} b^M_{-p-1-n+j} x.
\end{align*}
We define $*_p$ by
\begin{align}
a *_p b = \sum_{m=0}^p \binom{-p-1}{m} a_{[-p-1-m]}b.
\end{align}

Now we check that $(a *_p b)^M_0 x = a^M_0 b^M_0 x$. Indeed we have
\begin{align*}
(a *_p b)^M_0 x
= \sum_{j \in \Z_+} \sum_{m=0}^p (-1)^j \binom{-p-1}{m} \binom{-p-1-m}{j} a^M_{-m-j} b^M_{m+j} x.
\end{align*}
For an integer $\alpha$, such that $0 \leq \alpha \leq p$, the coefficient of $a^M_{-\alpha} b^M_{\alpha} x$ here is
\begin{align*}
\sum_{j=0}^{\alpha} (-1)^j \binom{-p-1}{\alpha-j} \binom{-p-1-\alpha+j}{j}
& = \sum_{j=0}^{\alpha} \binom{-p-1}{\alpha-j} \binom{p+\alpha}{j} \\
& = [\xi^{\alpha}]: (1+\xi)^{-p-1} (1+\xi)^{p+\alpha} \\
& = [\xi^{\alpha}]: (1+\xi)^{\alpha-1}
\end{align*}
which is $1$ if $\alpha=0$ and $0$ if $\alpha>0$. So indeed $(a *_p b)^M_0 = a^M_0 b^M_0$ and $M_p$ becomes a module over $(V_\G, *_p)$.

The algebra $(V_\G, *_p)$ is not associative, but recall equations (\ref{THcondition}) and (\ref{minus2ann}) which display elements $a \in V_\G$ such that $a^M_0|_{M_p} = 0$ for all $V$-modules $M$. Indeed, let
\begin{align*}
J_{p, \G} = \C\{(T + H)a | a \in V_\G\} + \C\{a_{[-2p-2+\chi(a, b)]}b | [\eps_a] + [\eps_b] = [0]\} \subseteq V_\G.
\end{align*}
From the remarks above we have
\begin{lemma} \label{Jannihilates}
\begin{align*}
(J_{p, \G})^M_0 M_p = 0.
\end{align*}
\end{lemma}

Later we shall see that $J_{p, \G}$ is a $2$-sided ideal of $(V_\G, *_p)$, and that the quotient
\[
\zhu_{p, \G}(V) = V_\G / J_{p, \G}
\]
is an associative algebra with unit element $\vac$. Because of the lemma above the action of $V_\G$ on $M_p$ factors to an action of $\zhu_{p, \G}(V)$.


\section{\texorpdfstring{The algebra $\zhu_{p, \G, \hbar}(V)$}{The Zhu algebra}}

\subsection{Formal definition} \label{defdef}

In section \ref{intro} we introduced $a_{[n]}b$. Observe that
\[
(1+z)^{\ga_a+p} Y(a, z) = \sum_{n \in \Z} a_{[n]} z^{-n-1}.
\]
Following \cite{DK} we put
\begin{align} \label{asndef}
\begin{split}
Z(a, z) &= (1+\hbar z)^{\ga_a+p} Y(a, z) = \sum_{n \in \Z} a_{[n, \hbar]} z^{-n-1}, \\
\text{so that} \quad a_{[n, \hbar]} &= \sum_{j \in \Z_+} \binom{\ga+p}{j} \hbar^j a_{(n+j)}.
\end{split}
\end{align}
Clearly $a_{[n, \hbar=1]} = a_{[n]}$. Motivated by the discussion in section \ref{intro} we put
\begin{align} \label{pthdef}
a *_{p, \hbar} b = \sum_{m = 0}^p \binom{-p-1}{m} \hbar^{-p-m} a_{[-p-1-m]} b.
\end{align}
We shall often drop $\hbar$ for simplicity of notation, when no confusion should arise. For example we write $a_{[n]}$ for $a_{[n, \hbar]}$ and $*_p$ for $*_{p, \hbar}$.

\begin{defn}
Let $V$ be a $\G/\Z$-graded vertex algebra with Hamiltonian operator $H$. The $\G$-twisted level $p$ Zhu algebra $\zhu_{p, \G, \hbar}(V)$ of $V$ is defined to be $V_{\G, \hbar} / J_{p, \G, \hbar}$ where $V_{\G, \hbar} = \C[\hbar, \hbar^{-1}] \otimes_\C V_\G$ and
\begin{align} \label{Jdef}
\begin{split}
V_{\G, \hbar}
\supseteq J_{p, \G, \hbar}
= {} & \C[\hbar, \hbar^{-1}] \{(T + \hbar H)a | [\eps_a] = [0]\} \\
& + \C[\hbar, \hbar^{-1}] \{a_{[-2p-2+\chi(a, b)]}b | [\eps_a] + [\eps_b] = [0]\}.
\end{split}
\end{align}
The product $*_p$ of $\zhu_{p, \G, \hbar}(V)$ is that inherited from $V_{\G, \hbar}$.
\end{defn}
This definition reduces to that of section \ref{intro} when $\hbar=1$.

The authors of \cite{DK} showed that the algebras $\zhu_{0, \G, \hbar}(V)$ are isomorphic for all $\hbar \in \C^\times$, but that there is an interesting degeneration at $\hbar = 0$. For $p > 0$ we use negative powers of $\hbar$ in the definition of $\zhu_{p, \G, \hbar}(V)$, so there is no obvious analogous degeneration. Nevertheless there is no harm in keeping $\hbar$.

In section \ref{borsec} we develop an analog of the Borcherds identity for the modified fields $Z(a, z)$ and use it to show that $J_{p, \G, \hbar}$ is a right ideal. In section \ref{sksmsec} we prove a skew-symmetry formula
\begin{align}
a *_p b - p(a, b) b *_p a = \hbar [a, b]_{\hbar} \pmod{J_{p, \G, \hbar}}
\end{align}
(where $[a, b]_{\hbar}$ is defined in that section). Then we use the skew-symmetry formula, in section \ref{derivprop}, to prove that $J_{p, \G, \hbar}$ is also a left ideal. In sections \ref{assocsec} and \ref{unisec} we show that, modulo $J_{p, \G, \hbar}$, the product $*_p$ is associative and $\vac$ is a unit, respectively.

\begin{lemma}\label{Tind} $\phantom{ }$
\begin{itemize}
\item For all $a, b \in V$, $n \in \Z$,
\begin{align} \label{Tindeqn}
(Ta)_{[n]}b + \hbar(\ga_a + p + n + 1) a_{[n]}b = -n a_{[n-1]}b.
\end{align}

\item $V_{[j]}V \subseteq V_{[k]}V$ whenever $j \leq k \leq -1$.
\end{itemize}
\end{lemma}

\begin{proof}
For the first part we have
\begin{align*}
Z(Ta, z)
&= (1 + \hbar z)^{\ga_a+1+p} Y(Ta, z)
= (1 + \hbar z)^{\ga_a+1+p} \partial_z Y(a, z) \\
&= (1 + \hbar z) \partial_z [(1 + \hbar z)^{\ga_a+p} Y(a, z)] - \hbar(\ga_a + p) (1 + \hbar z)^{\ga_a+p} Y(a, z) \\
&= [(1 + \hbar z) \partial_z - \hbar(\ga_a + p)] Z(a, z),
\end{align*}
equate coefficients of $z^{-n-1}$ to get (\ref{Tindeqn}). The second part follows immediately.
\end{proof}

Now suppose $a \in V_\G$, so that $\eps_a = 0$ and $\D_a = \ga_a$. From Lemma \ref{Tind} we have
\begin{align*}
[(T + \hbar H)a] *_p b
= {} & \sum_{m = 0}^p \binom{-p-1}{m} \hbar^{-p-m} [(T + \hbar H)a]_{[-p-1-m]} b \\
= {} & \sum_{m=0}^{p} \binom{-p-1}{m} \hbar^{-p-m} \left[(p+m+1) a_{[-p-2-m]}b + \hbar m a_{[-p-1-m]}b \right].
\end{align*}
If we put $m = n+1$ in the second summand and manipulate it slightly we can combine the sums to obtain
\begin{align}
[(T + \hbar H)a] *_p b = (2p+1) \binom{-p-1}{p} \hbar^{-2p} a_{[-2p-2]}b.
\end{align}
If $b \in V_\G$ too then $[\eps_a] + [\eps_b] = [0]$ and so $((T + \hbar H)V_\G) *_p V_\G \subseteq J_{p, \G, \hbar}$.

\begin{rmrk}
This calculation shows that in the untwisted case, where $\eps_a = 0$ for all $a$, the set $J_{p, \hbar}$ may be defined simply as the ideal in $(V, *_p)$ generated by $(T + H)V$ rather than the more complicated equation (\ref{Jdef}).
\end{rmrk}


\subsection{The Borcherds identity} \label{borsec}

Let $(V, Y, \vac)$ be a vertex algebra, recall \cite{Kac} the $n^\text{th}$ product of quantum fields:
\begin{align} \label{nthproddef}
\begin{split}
Y(a, w)_{(n)}Y(b, w)
= {} & \res_z \left[Y(a, z) Y(b, w) i_{z, w} (z-w)^n \right. \\
& \left. - p(a, b) Y(b, w) Y(a, z) i_{w, z} (z-w)^n \right].
\end{split}
\end{align}
Taking the residue $\res_z$ of the Borcherds identity shows that
\begin{align}\label{nthprodiden}
Y(a, w)_{(n)}Y(b, w) = Y(a_{(n)}b, w).
\end{align}
This is called the $n^\text{th}$ product identity. In this section we recast the $n^\text{th}$ product identity and the Borcherds identity, which are both in terms of $Y(a, w)$, in terms of $Z(a, w)$.

\begin{thm} \label{modn}
\begin{align}\label{modnthprodiden}
(1 + \hbar w)^{n + p + 1 - \chi(a, b)} Z(a_{[n]}b, w) = Z(a, w)_{(n)}Z(b, w).
\end{align}
Here the $n^\textrm{th}$ product of quantum fields is defined as above, i.e., equation (\ref{nthproddef}) but with $Z$ in place of $Y$.
\end{thm}

\begin{proof}
The left hand side is
\begin{align*}
(1 + \hbar w)^{n + p + 1 - \chi(a, b)} &  \sum_{j \in \Z_+} \binom{\ga_a + p}{j}\hbar^j Z(a_{(n+j)}b, w) \\
= & \sum_{j \in \Z_+} \binom{\ga_a + p}{j} \hbar^j (1 + \hbar w)^{\ga_a + \ga_b + 2p - j} Y(a_{(n+j)}b, w)
\end{align*}
(using $\ga_{a_{(k)}b} = \ga_a + \ga_b + \chi(a, b) - k - 1$). Now we use equation (\ref{nthprodiden}) to rewrite this as
\begin{align*}
\res_z \Biggl[ \sum_{j \in \Z_+} & \binom{\ga_a + p}{j} \hbar^j (z - w)^j (1 + \hbar w)^{\ga_a + \ga_b + 2p - j} \Biggr. \times \\
& \times \Biggl. \left\{Y(a, z)Y(b, w) i_{z, w}(z - w)^n - p(a, b) Y(b, w)Y(a, z) i_{w, z}(z - w)^n\right\} \Biggr].
\end{align*}
But
\begin{align*}
\sum_{j \in \Z_+} \binom{\ga_a + p}{j} \hbar^j (z - w)^j (1 + \hbar w)^{\ga_a + \ga_b + 2p - j} = (1 + \hbar z)^{\ga_a + p}(1 + \hbar w)^{\ga_b + p},
\end{align*}
so the left hand side of equation (\ref{modnthprodiden}) becomes
\begin{align*}
\res_z \left[ Z(a, z) Z(b, w) i_{z, w}(z - w)^n - p(a, b) Z(b, w) Z(a, z) i_{w, z}(z - w)^n \right],
\end{align*}
which is the right hand side.
\end{proof}

We use Theorem \ref{modn} to write down the modified analog of the Borcherds identity.
\begin{thm} \label{modbor}
For all $a, b, c \in V$,
\begin{align}
\begin{split}
\sum_{j \in \Z_+} & (1 + \hbar w)^{n+j+p+1-\chi(a, b)} Z(a_{[n+j]}b, w) c \, \partial_w^{(j)} \delta(z, w) \\
& = Z(a, z) Z(b, w) c \, i_{z, w}(z - w)^n - p(a, b) Z(b, w) Z(a, z) c \, i_{w, z}(z - w)^n.
\end{split} \label{defbor}
\end{align}
\end{thm}

\begin{proof}
The theorem and its proof are essentially the same as Theorem 2.3 of \cite{DK} and its proof.
\end{proof}

We use Theorem \ref{modbor} to obtain an expression for $(a_{[n]}b)_{[k]}c$. We begin by extracting the $z^{-m-1}$ coefficient, obtaining
\begin{align*}
& \sum_{j \in \Z_+} \binom{m}{j} (1 + \hbar w)^{n+j+p+1-\chi(a, b)} Z(a_{[n+j]}b, w) c w^{m-j} \\
& \phantom{\lim} = \sum_{j \in \Z_+} (-1)^j \binom{n}{j} \left[ a_{[m+n-j]}\left(Z(b, w)c \right) w^j - p(a, b) (-1)^n Z(b, w)(a_{[m+j]}c)w^{n-j} \right].
\end{align*}
Next we multiply through by $(1 + \hbar w)^{-n-p-1+\chi(a, b)}$, expand $(1+\hbar w)^{-n-p-1+\chi(a, b)}$ in positive powers of $w$, and extract the $w^{-k-1}$ coefficient, obtaining
\begin{align*}
& \sum_{i, j \in \Z_+} \binom{m}{j} \binom{j}{i} \hbar^i (a_{[n+j]}b)_{[i+k+m-j]}c \\
& \phantom{\lim} = \sum_{i, j \in \Z_+} (-1)^j \hbar^i \binom{n}{j} \binom{-n-p-1+\chi(a, b)}{i} \times \\
& \phantom{\lim \lim} \times \Biggl[ a_{[m+n-j]}\left(b_{[i+j+k]}c \right) - p(a, b) (-1)^n b_{[i-j+n+k]}(a_{[m+j]}c) \Biggr].
\end{align*}

Finally we put $m = 0$, to obtain
\begin{align}
\begin{split}
(a_{[n]}b)_{[k]}c
& = \sum_{i, j \in \Z_+} (-1)^j \hbar^i \binom{n}{j} \binom{-n-p-1+\chi(a, b)}{i} \times \\
& \phantom{\lim} \times \Biggl[ a_{[n-j]}\left(b_{[i+j+k]}c \right) - p(a, b) (-1)^n b_{[i-j+n+k]}(a_{[j]}c) \Biggr].
\end{split} \label{expbor}
\end{align}

If $n, k \leq -p-1$, then $-n-p-1+\chi(a, b) \geq 0$ and the sum on the right hand side of (\ref{expbor}) runs over $0 \leq i \leq -n-p-1+\chi(a, b)$. Therefore $i-j+n+k \leq -2p-2+\chi(a, b)$ and the $b_{[\cdot]}(a_{[\cdot]}c)$ terms all lie in $J_{p, \G, \hbar}$. If we put $n = -2p-2+\chi(a, b)$ and $k \leq -p-1$ we see that
\begin{align*}
& (a_{[-2p-2+\chi(a, b)]}b)_{[k]}c \\
& \equiv \sum_{i, j \in \Z_+} (-1)^j \binom{-2p-2+\chi(a, b)}{j} \binom{p+1}{i} \hbar^i a_{[-2p-2+\chi(a, b)-j]}\left(b_{[i+j+k]}c\right) \pmod{J_{p, \G, \hbar}}.
\end{align*}
This implies that $\C[\hbar, \hbar^{-1}] \{a_{[-2p-2-\chi(a, b)]}b | [\eps_a] + [\eps_b] = [0]\}$ is a right ideal with respect to the products $a_{[k]}b$ for $k \leq -p-1$, and thus a right ideal with respect to $*_p$. Combining this with the calculation at the end of section \ref{defdef} shows that $J_{p, \G, \hbar}$ is a right ideal of $(V_\G, *_p)$.


\subsection{\texorpdfstring{The skew-symmetry formula for $Z(a, w)$}{The skew-symmetry formula for Z(a, w)}} \label{sksmsec}

Recall the skew-symmetry formula \cite{Kac},
\begin{align*}
Y(b, z) a = p(a, b) e^{zT} Y(a, -z) b
\end{align*}
for all $a, b \in V$. In this section we recast this in terms of the modified field $Z(a, w)$.

If $c \in V_\G$ then $T c \equiv \hbar (-\D_c) c \pmod{J_{p, \G, \hbar}}$ by the definition of $J_{p, \G, \hbar}$, hence
\begin{align}
T^{(k)}c
&\equiv \hbar^k \binom{-\D_c}{k} c \nonumber \\
\mbox{and} \qquad
e^{zT} c
&\equiv \sum_{j \in \Z_+} \hbar^j \binom{-\D_c}{j} c z^j
= (1 + \hbar z)^{-\D_c} c. \label{ezt}
\end{align}
Here, and further, we often write $\equiv$ for $\equiv \text{mod} J_{p, \G, \hbar}$.

Assume now that $[\eps_a] + [\eps_b] = [0]$; this implies that all $a_{(n)}b$ and $b_{(n)}a$ lie in $V_\G$. From (\ref{ezt}) we have
\begin{align*}
e^{zT} Y(a, -z)b
& = \sum_{n \in \Z} (-z)^{-n-1} e^{zT}(a_{(n)}b) \\
& \equiv (1 + \hbar z)^{-\D_a - \D_b} \sum_{n \in \Z} (-z)^{-n-1} (1 + \hbar z)^{n+1} a_{(n)} b \\
& = (1 + \hbar z)^{-\D_a - \D_b} Y\left(a, \frac{-z}{1 + \hbar z}\right) b
= (1 + \hbar z)^{\eps_a - \D_b + p} Z\left(a, \frac{-z}{1 + \hbar z}\right) b.
\end{align*}
Therefore
\begin{align}
Z(b, z) a
&= p(a, b) (1 + \hbar z)^{\D_b + \eps_b + p} e^{zT} Y(a, -z) b \nonumber \\
&\equiv p(a, b) (1 + \hbar z)^{p-\ga_a+\eps_a+\eps_b} Y\left(a, \frac{-z}{1 + \hbar z}\right) b \label{Yswap} \\
&\equiv p(a, b) (1 + \hbar z)^{2p+\eps_a+\eps_b} Z\left(a, \frac{-z}{1 + \hbar z}\right) b. \label{swap}
\end{align}

We expand (\ref{Yswap}) and (\ref{swap}) and equate coefficients to obtain the equations:
\begin{align}
b_{[n]}a &\equiv p(a, b) \sum_{j \in \Z_+} (-1)^{n+j+1} \binom{(p-\ga_a+\eps_a+\eps_b)+1+n+j}{j} \hbar^j a_{(n+j)}b, \label{swapduple} \\
\mbox{and} \quad b_{[n]}a &\equiv p(a, b) \sum_{j \in \Z_+} (-1)^{n+j+1} \binom{2p+\eps_a+\eps_b+1+n+j}{j} \hbar^j a_{[n+j]}b, \label{rough}
\end{align}
respectively. We shall use (\ref{swapduple}) in appendix \ref{appcalc}, we need it there with $\eps_a \neq 0$. For now we use (\ref{rough}) and we only need it for $a, b \in V_\G$. Let us write (\ref{rough}) in the slightly different form:
\begin{align} \label{swapexpl}
b_{[-p-1-m]}a
\equiv p(a, b) (-1)^{p-m} \sum_{j \in \Z_+} \binom{p-m+j}{j} (-\hbar)^j a_{[-p-1-m+j]}b,
\end{align}
for $a, b \in V_\G$.

Substituting (\ref{swapexpl}) into the definition of $b *_p a$ yields
\begin{align} \label{baexp}
b *_p a
\equiv p(a, b) \sum_{m=0}^p \sum_{j \in \Z_+} \binom{-p-1}{m} \binom{p-m+j}{j} (-\hbar)^{-p-m+j} a_{[-p-1-m+j]}b.
\end{align}
For $\alpha \in \Z$, let us consider the coefficient of $a_{[-p-1-\alpha]}b$ in (\ref{baexp}). If $\alpha > p$ the coefficient is $0$. If $\alpha \leq p$ the coefficient is
\begin{align*}
p(a, b) (-\hbar)^{-p-\alpha} \sum_{j \in \Z_+} \binom{-p-1}{\alpha+j} \binom{p-\alpha}{j}
= p(a, b) (-\hbar)^{-p-\alpha} \binom{-\alpha-1}{p},
\end{align*}
where we have used
\begin{align*}
\sum_{j \in \Z_+} \binom{-p-1}{\alpha+j} \binom{p-\alpha}{j}
& = [\xi^{\alpha}]: (1+\xi)^{-p-1}(1+\xi^{-1})^{p-\alpha}
= [\xi^p]: (1+\xi)^{-p-1} (\xi+1)^{p-\alpha} \\
& = [\xi^p]: (1+\xi)^{-\alpha-1}
= \binom{-\alpha-1}{p}.
\end{align*}

The terms in equation (\ref{baexp}) with $0 \leq \alpha \leq p$ may be gathered together and, using $(-1)^n \binom{-m-1}{n} = (-1)^m \binom{-n-1}{m}$ for $m, n \in \Z_+$, reduced to $p(a, b) a *_p b$. The sum of the remaining terms (i.e., those with $\alpha < 0$) is
\begin{align*}
p(a, b) \sum_{\alpha < 0} (-\hbar)^{-p-\alpha} \binom{-\alpha-1}{p} a_{[-p-1-\alpha]}b
&= -\hbar p(a, b) \sum_{k \in \Z_+} (-\hbar)^{k} \binom{p+k}{p} a_{[k]}b \\
&= -\hbar p(a, b) \sum_{k \in \Z_+} \hbar^k \binom{-p-1}{k} a_{[k]}b \\
&= -\hbar p(a, b) [a, b]_{\hbar},
\end{align*}
where $k = -p-1-\alpha$ and
\begin{align} \label{hbarbrak}
\begin{split}
[a, b]_{\hbar}
& = \res_z (1 + \hbar z)^{-p-1} Z(a, z) b
= \sum_{j \in \Z_+} \binom{-p-1}{j} \hbar^j a_{[j]}b \\
& = \res_z (1 + \hbar z)^{\ga_a - 1} Y(a, z) b
= \sum_{j \in \Z_+} \binom{\ga_a - 1}{j} \hbar^j a_{(j)}b.
\end{split}
\end{align}

We have proved that if $a, b \in V_\G$, then
\begin{align}
a *_p b - p(a, b) b *_p a \equiv \hbar [a, b]_{\hbar} \pmod{J_{p, \G, \hbar}}. \label{skewsymm}
\end{align}
We call equation (\ref{skewsymm}) the skew-symmetry formula.


\subsection{\texorpdfstring{The bracket $[\cdot, \cdot]_{\hbar}$}{The bracket}} \label{derivprop}

In this section we prove that $J_{p, \G, \hbar}$ is a left ideal of $(V_{\G, \hbar}, *_p)$. This requires us to first prove some identities for $[\cdot, \cdot]_{\hbar}$.
\begin{lemma} \label{brakderivlemma}
If $a, b_{[n]}c \in V_\G$, (so that $\eps_a = 0$ and $[\eps_b] + [\eps_c] = [0]$), then
\begin{align}
[a, b_{[n]}c]_{\hbar} = ([a, b]_{\hbar})_{[n]} c + p(a, b) b_{[n]}([a, c]_{\hbar}). \label{brakderiv}
\end{align}
\end{lemma}

\begin{proof}
We begin with
\begin{align*}
\left[a, \left( Z(b, w)c \right)\right]_{\hbar}
= {} & \res_z (1 + \hbar z)^{-p-1} Z(a, z) Z(b, w) c \\
= {} & p(a, b) \res_z (1 + \hbar z)^{-p-1} Z(b, w) Z(a, z) c \\
& + \res_z (1 + \hbar z)^{-p-1} [Z(a, z), Z(b, w)] c \\
= {} & p(a, b) Z(b, w) [a, c]_{\hbar} + \res_z (1 + \hbar z)^{-p-1} [Z(a, z), Z(b, w)]c.
\end{align*}

Now we use (\ref{defbor}) with $n = 0$ to expand the second term in the last line; it becomes
\begin{align*}
& \res_z (1 + \hbar z)^{-p-1} \sum_{j \in \Z_+} (1 + \hbar w)^{p+1+j-\chi(a, b)} Z(a_{[j]}b, w) \partial_w^{(j)}(z, w) \\
&= \sum_{j \in \Z_+} (1 + \hbar w)^{p+1+j-\chi(a, b)} Z(a_{[j]}b, w) \res_z (1 + \hbar z)^{-p-1} \partial_w^{(j)}(z, w) \\
&= \sum_{j \in \Z_+} (1 + \hbar w)^{p+1+j-\chi(a, b)} Z(a_{[j]}b, w) \binom{-p-1}{j} \hbar^j (1 + \hbar w)^{-p-1-j} \\
&= (1 + \hbar w)^{-\chi(a, b)} Z([a, b]_{\hbar}, w) c.
\end{align*}
So we have proved that
\begin{align*}
[a, (Z(b, w) c)]_{\hbar} = p(a, b) Z(b, w) [a, c]_{\hbar} + (1 + \hbar w)^{-\chi(a, b)} Z([a, b]_{\hbar}, w) c.
\end{align*}
If $\eps_a = 0$, then $\chi(a, a') = 0$ for all $a' \in V$. Extracting the $w^{-n-1}$ coefficient yields equation (\ref{brakderiv}).
\end{proof}

Let $a, b_{[n]}c \in V_\G$, observe that $\chi(a_{(k)}b, c) = \chi(b, a_{(k)}c) = \chi(b, c)$. Using (\ref{skewsymm}), (\ref{brakderiv}) and the fact that $J_{p, \G, \hbar}$ is a right ideal of $V_{\G, \hbar}$, we have
\begin{align*}
a *_p (b_{[-2p-2+\chi(b, c)]}c)
\equiv {} & p(a, b) p(a, c) (b_{[-2p-2+\chi(b, c)]}c) *_p a + \hbar [a, (b_{[-2p-2+\chi(b, c)]}c)]_{\hbar} \\
\equiv {} & p(a, b) p(a, c) (b_{[-2p-2+\chi(b, c)]}c) *_p a \\
&+ \hbar ([a, b]_{\hbar})_{[-2p-2+\chi(b, c)]}c + p(a, b) \hbar b_{[-2p-2+\chi(b, c)]}([a, c]_{\hbar}) \\
\equiv {} & 0.
\end{align*}
Next observe that if $a, b \in V_\G$ then
\begin{align*}
[(T + \hbar H)a, b]_{\hbar}
&= \res_z (1 + \hbar z)^{\ga_a - 1} [(Ta)(z) + \hbar \ga_a a(z)]b \\
&= \res_z \partial_z \left[(1 + \hbar z)^{\ga_a} a(z)b\right] = 0.
\end{align*}
Therefore we have
\begin{align*}
b *_p [(T + \hbar H)a]
&\equiv p(a, b) [(T + \hbar H)a] *_p b - p(a, b) \hbar [(T + \hbar H)a, b]_{\hbar} \\
&\equiv 0.
\end{align*}

From the definition of $J_{p, \G, \hbar}$ in equation (\ref{Jdef}), and the remarks above we see that $(V_\G) *_p (J_{p, \G, \hbar}) \subseteq J_{p, \G, \hbar}$, i.e., $J_{p, \G, \hbar}$ is a left ideal in $(V_{\G, \hbar}, *_p)$. Since $J_{p, \G, \hbar}$ is also a right ideal, it is a $2$-sided ideal.


\subsection{\texorpdfstring{$\zhu_{p, \G, \hbar}(V)$ is associative}{Associativity}} \label{assocsec}

Let $a, b, c \in V_\G$, we use (\ref{expbor}) to expand $(a *_p b) *_p c$ directly and prove it equals $a *_p (b *_p c)$ modulo $J_{p, \G, \hbar}$. For brevity we put
\begin{align*}
D_r(b, c) = \sum_{m=0}^p \binom{-p-1}{m} \hbar^{-p-m} b_{[-p-1-m+r]}c,
\end{align*}
in particular, $D_0(b, c) = b *_p c$.

We expand $(a *_p b) *_p c$ using the definition (\ref{pthdef}), then we apply equation (\ref{expbor}) to $(a_{[-p-1-m]}b)_{[-p-1-n]}c$ to obtain
\begin{align}
(a *_p b) *_p c
& = \sum_{m=0}^p \sum_{n=0}^p \binom{-p-1}{m} \binom{-p-1}{n} \hbar^{-p-m} \hbar^{-p-n} (a_{[-p-1-m]}b)_{[-p-1-n]}c \nonumber \\
& \equiv \sum_{m=0}^p \sum_{i, j \in \Z_+} (-1)^j \binom{-p-1}{m} \binom{-p-1-m}{j} \binom{m}{i} \hbar^{-p-m+i} a_{[-p-1-m-j]}D_{i+j}. \label{assocorig}
\end{align}
We have omitted terms of the form $b_{[\cdot]}(a_{[\cdot]}c)$ here, because they lie in $J_{p, \G, \hbar}$ (see the remarks following equation (\ref{expbor})).

We change indices to $\beta = m+j$ and $\gamma = m-i$. Modulo $J_{p, \G, \hbar}$, the sum (\ref{assocorig}) becomes
\begin{align} \label{associnter}
\sum_{0 \leq \gamma \leq m \leq \beta \leq p} (-1)^{\beta-m} \binom{-p-1}{m} \binom{-p-1-m}{\beta-m} \binom{m}{\gamma} \hbar^{-p-\gamma} a_{[-p-1-\beta]}D_{\beta-\gamma}.
\end{align}
The range of the indices in the summation (\ref{associnter}) deserves explanation. Since $i \geq 0$, we have $\ga = m - i \leq m$. Since $j \geq 0$, we have $\beta = m + j \geq m$. Since $\binom{m}{\gamma}$ appears in the summand, $\gamma \geq 0$. Terms with $\beta > p$ lie in $J_{p, \G, \hbar}$, so $\beta \leq p$. Thus the summation is over the range $0 \leq \gamma \leq m \leq \beta \leq p$.

Fix $\gamma, \beta \in \Z$ such that $0 \leq \gamma \leq \beta \leq p$. Then the coefficient of $\hbar^{-p-\gamma} a_{[-p-1-\beta]}D_{\beta-\gamma}$ in (\ref{associnter}) is
\begin{align*}
& \sum_{m \in \Z} \binom{-p-1}{m} \binom{p+\beta}{\beta-m} \binom{m}{\gamma} \\
= & \sum_{m \in \Z} \frac{(-p-1)\cdots(-p-m)}{m!} \cdot \frac{(p+\beta)!}{(\beta-m)!(p+m)!} \cdot \frac{m!}{\gamma!(m-\gamma)!} \\
= & \sum_{m \in \Z} (-1)^m \frac{(p+\beta)!}{(\beta-m)!p!} \cdot \frac{1}{\gamma!(m-\gamma)!} \\
= & \frac{(p+\beta)!}{p!\gamma!(\beta-\gamma)!} \sum_{m \in \Z} (-1)^m \binom{\beta-\gamma}{\beta-m}.
\end{align*}
The final sum here is just the sum of all the coefficients in the expansion of $(-1)^{\beta} (1-\xi)^{\beta-\gamma}$, which is $0$ if $\gamma < \beta$ and $1$ if $\gamma = \beta$. In the latter case we obtain
\begin{align*}
(-1)^{\beta} \frac{(p+\beta)!}{p!\beta!} = \binom{-p-1}{\beta}
\end{align*}
as the coefficient. Therefore 
\begin{align*}
(a *_p b) *_p c
&\equiv \sum_{\beta=0}^p \binom{-p-1}{\beta} \hbar^{-p-\beta} a_{[-p-1-\beta]}(b *_p c) \pmod{J_{p, \G, \hbar}} \\
&= a *_p (b *_p c).
\end{align*}


\subsection{\texorpdfstring{$\zhu_{p, \G, \hbar}(V)$ is unital}{Unitality}} \label{unisec}

By the vacuum axiom, we have $\vac_{(n)}a = \delta_{n, -1}a$. Therefore
\begin{align*}
\vac_{[n]}a = \sum_{j \in \Z_+} \binom{p}{j} \hbar^j \vac_{(n+j)}a = \binom{p}{-n-1} \hbar^{-n-1} a,
\end{align*}
and hence
\begin{align*}
\vac *_p a
&= \sum_{m=0}^p \binom{-p-1}{m} \hbar^{-p-m} \vac_{[-p-1-m]}a \\
&= \sum_{m=0}^p \binom{-p-1}{m} \binom{p}{p+m} a = a.
\end{align*}
Note that $[\vac, a]_{\hbar} = 0$, so by skew-symmetry
\begin{align*}
a *_p \vac \equiv \vac *_p a \equiv a \pmod{J_{p, \G, \hbar}}
\end{align*}
as well.

It is possible that $J_{p, \G, \hbar} = V_\G$; in this case $\zhu_{p, \G, \hbar}(V) = 0$. Suppose $V$ has a nonzero $(\G, H)$-twisted positive energy module $M$, we may assume that $M_0 \neq 0$ without loss of generality. The identity element $[\vac] \in \zhu_{0, \G, \hbar=1}(V)$ has nonzero action on $M_0$; hence $\zhu_{0, \G, \hbar}(V) \neq 0$. The higher level Zhu algebras are all quotients of $V_\G$ by smaller ideals (see the next section), and so are also nonzero.


\section{\texorpdfstring{Homomorphisms between different Zhu algebras of $V$}{Homomorphisms between different Zhu algebras of V}} \label{maps}

To avoid confusion we write the level $p$ $n^{\text{th}}$ product as $a_{[n, p]}b$ in this section. Because
\begin{align*}
a_{[n, p]}b = \res_z z^n (1 + \hbar z)^{\ga_a + p} Y(a, z) b,
\end{align*}
we have
\begin{align*}
a_{[n, p]}b = a_{[n, p-1]}b + \hbar a_{[n+1, p-1]}b.
\end{align*}
This implies $J_{p, \G, \hbar} \subseteq J_{p-1, \G, \hbar}$. Furthermore
\begin{align*}
a *_p b
&= \sum_{m=0}^p \binom{-p-1}{m} \hbar^{-p-m} a_{[-p-1-m, p]}b \\
&= \sum_{m=0}^p \binom{-p-1}{m} \hbar^{-p-m} a_{[-p-1-m, p-1]}b + \sum_{m=0}^p \binom{-p-1}{m} \hbar^{-p-m+1} a_{[-p-m, p-1]}b \\
&= \sum_{n=1}^{p+1} \binom{-p-1}{n-1} \hbar^{-p-n+1} a_{[-p-n, p-1]}b + \sum_{m=0}^p \binom{-p-1}{m} \hbar^{-p-m+1} a_{[-p-m, p-1]}b \\
&\equiv \sum_{m=0}^p \binom{-p}{m} \hbar^{-p-m+1} a_{[-p-m, p-1]}b \pmod{J_{p-1, \G, \hbar}} \\
&= a *_{p-1} b.
\end{align*}

Hence, for all $p \geq 1$, the identity map on $V$ induces a surjective homomorphism of associative algebras $\phi_p : \zhu_{p, \G, \hbar}(V) \twoheadrightarrow \zhu_{p-1, \G, \hbar}(V)$, i.e., the diagram
\begin{align} \label{factorcomm}
\xymatrix{
V_{\G, \hbar} \ar@{->>}[d]_{\pi_p} \ar@{->>}[dr]^{\pi_{p-1}} \\
\zhu_{p, \G, \hbar}(V) \ar@{->>}[r]_{\phi_p} & \zhu_{p-1, \G, \hbar}(V) \\
}
\end{align}
commutes.


\section{Representation theory} \label{repnthry}


In this section let us fix $\hbar=1$ and write $\zhu_{p, \G}(V) = \zhu_{p, \G, \hbar=1}(V)$. When discussing categories of modules we write $\perep(V)$ for the category of $(\G, H)$-twisted positive energy $V$-modules. Morphisms in $\perep(V)$ are parity-preserving linear maps $f : M_1 \rightarrow M_2$ such that
\begin{itemize}
\item $f(a^{M_1}_{(n)}x) = a^{M_2}_{(n)}f(x)$ for all $a \in V$, $x \in M_1$, $n \in [\ga_a]$

\item $\deg f(x) = \deg x$ for all $x \in M$.
\end{itemize}
If the degrees of all elements in a module are shifted by a fixed amount then the resulting module is essentially the same, but it is convenient for us not to identify such modules.

\subsection{\texorpdfstring{The Restriction Functor $\Omega_p$}{The Restriction Functor}} \label{restrictfunct}

\begin{defn}[Restriction functor $\Omega_p$]
For $M \in \perep(V)$ let $\Omega_p(M) = M_p$, endowed with an action of $\zhu_{p, \G}(V)$ via $[a] x = a_0^M x$ (cf. section \ref{motivation}). If $f : M \rightarrow M'$ is a morphism in $\perep(V)$ then let $\Omega_p(f) = f|_{M_p}$. $\Omega_p$ is a functor from the category $\perep(V)$ to the category $\rep(\zhu_{p, \G}(V))$ of $\zhu_{p, \G}(V)$-modules.
\end{defn}

The following lemma is adapted from \cite{Li}, Lemma 6.1.1.
\begin{lemma} \label{lithesis}
Let $V$ be a $\G/\Z$-graded vertex algebra, and let $M \in \perep(V)$. Fix $a, b \in V$ of homogeneous conformal weight, $x \in M$, $m \in [\eps_a]$, and $k \in [\eps_b]$. There exists $c \in V$ such that
\begin{align*}
a^M_m(b^M_k x) = c^M_{m+k} x
\end{align*}
(however $c$ need not be of homogeneous conformal weight).
\end{lemma}

\begin{proof}
Consider equation (\ref{modconbor}) with $a$, $b$, $x$, $m$, $k$ as above, and $n \in \Z$. Because $Y^M(a, z)$ and $Y^M(b, z)$ are quantum fields, there exist $\overline{m} \in [\eps_a]$ and $\overline{k} \in [\eps_b]$ such that $a^M_m x = b^M_k x = 0$ for $m \geq \overline{m}$, $k \geq \overline{k}$. The lemma obviously holds when $k \geq \overline{k}$.

Substitute $m = \overline{m}$ into (\ref{modconbor}) to get
\begin{align*}
\sum_{j \in \Z_+} \binom{\overline{m}+\D_a-1}{j} (a_{(n+j)}b)^M_{\overline{m}+k} x
= \sum_{j \in \Z_+} (-1)^j\binom{n}{j} a^M_{\overline{m}+n-j} b^M_{k+j-n} x,
\end{align*}
then put $k = \overline{k}+n-1$ to obtain
\begin{align*}
\sum_{j \in \Z_+} \binom{\overline{m}+\D_a-1}{j} (a_{(n+j)}b)^M_{\overline{m}+\overline{k}+n-1} x
= a^M_{\overline{m}+n} b^M_{\overline{k}-1} x.
\end{align*}
Hence the lemma is true for $k = \overline{k}-1$ too, we simply let
\[
c = \sum_{j \in \Z_+} \binom{\overline{m}+\D_a-1}{j} a_{(m-\overline{m}+j)}b.
\]

Now put $m = \overline{m}$ and $k = \overline{k}+n-2$ to obtain
\begin{align*}
\sum_{j \in \Z_+} \binom{\overline{m}+\D_a-1}{j} (a_{(n+j)}b)^M_{\overline{m}+\overline{k}+n-2} x
= a^M_{\overline{m}+n} b^M_{\overline{k}-2} x - n a^M_{\overline{m}+n-1} b^M_{\overline{k}-1} x.
\end{align*}
Since we can write $a^M_{\overline{m}+n-1} b^M_{\overline{k}-1} x$ as $c^M_{\overline{m}+\overline{k}+n-2} x$, we can now do the same for $a^M_{\overline{m}+n} b^M_{\overline{k}-2} x$.

The general case follows inductively, we write any term of the form $a^M_m b^M_k x$ as a linear combination of terms of the form $(c^i)^M_{r_i} x$. The graded structure of $M$ implies that for all such terms $r_i = m+k$, so the linear combination can be taken to be the single term $c = \sum_i c^i$.
\end{proof}

\begin{prop} \label{smallirred}
If $M \in \perep(V)$ is irreducible then $\Omega_p(M)$ is either $0$ or it is irreducible.
\end{prop}

\begin{proof}
Suppose to the contrary that $N = \Omega_p(M)$ has a proper $\zhu_{p, \G}(V)$-submodule $N'$. From Lemma \ref{lithesis} we see that the $V$-submodule $M'$ of $M$ generated by $N'$ is the span of elements $a^M_n x$ where $a \in V$, $n \in [\eps_a]$, and $x \in N$. But then
\begin{align*}
M'_p = (V_\G)_0 N' = \zhu_{p, \G}(V) N' = N' \subsetneq N = M_p,
\end{align*}
which contradicts the irreducibility of $M$.
\end{proof}

\begin{rmrk} \label{QleqP}
The remarks that led to Lemma \ref{Jannihilates} also imply that $J_{p, \G} M_q = 0$ for all $q$ such that $0 \leq q \leq p$. Thus, for all such $q$, $M_q$ is a $\zhu_{p, \G}(V)$-module.

An equivalent proof of this fact comes from the surjective homomorphism $\phi_{j+1}$, which allows us to regard $M_j$ as a $\zhu_{j+1, \G}(V)$-module. By composing these homomorphisms we may regard $M_j$ as a $\zhu_{p, \G}(V)$-module, for each $p \geq j$. This argument only works for $j \in \Z_+$, but in appendix \ref{pnotinZ} we define Zhu algebras for non-integer $p$; with these in hand we can deduce that $M_q$ is a $\zhu_{p, \G}(V)$-module whenever $q \leq p$.
\end{rmrk}

\begin{rmrk} \label{inequivrem}
Let $M$ be irreducible and $0 \leq q \leq p$. The $\zhu_{p, \G}(V)$-module $M_q$ is either $0$ or irreducible for the same reason as in Proposition \ref{smallirred}. Recall that if $\omega \in V$ is a Virasoro element and $L(z) = Y(\omega, z)$, then $H = L_0$ is an energy operator on $V$. From equation (\ref{Hcomm}) we have $L^M_0 : M_0 \rightarrow M_0$ commuting with $a_0$ for all $a \in V_\G$. We assume $V$, and therefore $M_0$, has countable dimension. Then, by the Dixmier-Schur lemma, $L^M_0|_{M_0} = h I_{M_0}$ for some constant $h \in \C$. Let $M_j \neq 0$. Because of Lemma \ref{lithesis}, and the fact that $M$ is irreducible, there exists $a \in V$ such that $a^M_{-j}|_{M_0} : M_0 \rightarrow M_j$ is nonzero. Using (\ref{Hcomm}) again shows that $L^M_0|_{M_j} = (h+j) I_{M_j}$. The nonzero $\zhu_{p, \G}(V)$-modules $M_q$, for $0 \leq q \leq p$, are pairwise non-isomorphic, because $L^M_0$ acts diagonally on each with distinct eigenvalues.
\end{rmrk}


\subsection{\texorpdfstring{The Induction Functors $M^p$ and $L^p$}{Two Induction Functors}} \label{inducfunct}

Let us fix a $\zhu_{p, \G}(V)$-module $N$. In this section we construct (in a functorial manner) a module $M \in \perep(V)$ such that $M_p = N$.

\begin{lemma} \label{liezhu}
The linear map $\varphi_p : (\lie{V})_0 \rightarrow \zhu_{p, \G}(V)$ defined by $a_0 \mapsto [a]$ is well defined and is a surjective homomorphism of Lie superalgebras, where $(\lie{V})_0$ has the Lie bracket (\ref{conbrak}), and $\zhu_{p, \G}(V)$ has the commutator bracket $[a, b] = a *_p b - p(a, b) b *_p a$.

Furthermore, the following diagram (of Lie superalgebra homomorphisms) commutes:
\begin{align} \label{liecomm}
\begin{split}
\xymatrix{
(\lie V)_0 \ar@{->>}[d]_{\varphi_p} \ar@{->>}[dr]^{\varphi_{p-1}} \\
\zhu_{p, \G}(V) \ar@{->>}[r]_{\phi_p} & \zhu_{p-1, \G}(V). \\
}
\end{split}
\end{align}
\end{lemma}

\begin{proof}
Recall definition (\ref{lieVdef}) of the Lie superalgebra $\lie V$ as the quotient $Q / (T + \partial_t) Q$ where $Q$ is spanned by elements of the form $at^m$, $m \in [\ga_a]$. Let $f : V_\G \rightarrow (\lie V)_0$ be defined by $a \mapsto a_0$. The degree $0$ piece of $(T + \partial_t) Q$ is spanned by elements of the form $(Ta)_0 + \D_a a_0$ where $a \in V_\G$, hence it is contained in $f(J_{p, \G})$. Therefore the canonical surjection $V_\G \twoheadrightarrow \zhu_{p, \G}(V)$ factors to give a linear map $\varphi_p: (\lie{V})_0 \twoheadrightarrow \zhu_{p, \G}(V)$.  

Let $a, b \in V_\G$. Equation (\ref{conbrak}) with $m = n = 0$ is
\begin{align*}
[a_0, b_0] = \sum_{j \in \Z_+} \binom{\D_a-1}{j} (a_{(j)}b)_0.
\end{align*}
Equation (\ref{hbarbrak}) with $\hbar=1$ is
\begin{align*}
[a, b]_{1} = \sum_{j \in \Z_+} \binom{\D_a - 1}{j} a_{(j)}b.
\end{align*}
Equation (\ref{skewsymm}) with $\hbar=1$ is
\begin{align*}
a *_p b - p(a, b) b *_p a \equiv [a, b]_{1} \pmod{J_{p, \G}}.
\end{align*}
These combine to imply that $\varphi_p([a_0, b_0]) = a *_p b - p(a, b) b *_p a$, i.e., that $\varphi_p$ is a Lie superalgebra homomorphism.

Diagram (\ref{liecomm}) is commutative because the maps in question are all induced from the identity map on $V_\G$.
\end{proof}

We now write $\g = \lie V$. Because of Lemma \ref{liezhu} our $\zhu_{p, \G}(V)$-module $N$ is naturally a $\g_0$-module. Consider the graded Lie subalgebra $\g_+ = \g_0 + \g_{>p}$ where $\g_{>p} = \oplus_{j>p} \g_j$. We extend the representation of $\g_0$ on $N$ to a representation of $\g_+$ by letting $\g_{>p}$ act by $0$. Then we induce from $\g_+$ to $\g$ to get the $\g$-module
\[
\tilde{M} = \ind_{\g_+}^\g N = U(\g) \otimes_{U(\g_+)} N.
\]
We make $\tilde{M}$ into a graded $\g$-module by declaring that $\deg N = p$ and that $a_n$ lowers degree by $n$.

We write $a^M_n$ for the image of $a_n$ in $\en \tilde{M}$, and we put
\begin{align*}
Y^M(a, z) = \sum_{n \in [\eps_a]} a^M_{n}z^{-n-\Delta_a} \in (\en{M})[[z, z^{-1}]]z^{-\ga_a}.
\end{align*}
We claim $Y^M(a, z)$ is a quantum field, the proof is by induction. By construction $a^M_n x = 0$ whenever $n > p$ and $x \in N$, this is the base case. Fix $s \in \Z_+$ and suppose that the claim holds for all length $s$ monomials
\begin{align*}
\mathbf{m} = (b^1)^M_{n_1} (b^2)^M_{n_2} \cdots (b^s)^M_{n_s} x,
\end{align*}
i.e., for all $c \in V$ we have $c^M_n \mathbf{m} = 0$ for $n \gg 0$. Let $a, b \in V$ and $k \in [\eps_b]$ be fixed now. By equation (\ref{conbrak}) we have
\begin{align*}
a^M_m b^M_k \mathbf{m} = b^M_k a^M_m \mathbf{m} + \sum_{j \in \Z_+} \binom{m+\D_a-1}{j} (a_{(j)}b)^M_{m+k} \mathbf{m},
\end{align*}
where the sum is finite because $Y(a, z)$ is a quantum field. By the inductive assumption, each of these finitely many terms vanishes for $m \gg 0$. Hence the claim holds for length $s+1$ monomials.

Let $F_r = \oplus_{j \geq r} \g_j$ for $r \in \overline{\G} \cap \R_+$. We declare the subsets $\{U(\g)F_r\}_{r \geq 0}$ to be a fundamental system of neighborhoods of $0$. The multiplication in $U(\g)$ is continuous with respect to this topology, so $U(\g)$ becomes a topological algebra. Let
\[
\hat{U} = \varprojlim (U(\g) / U(\g)F_r)
\]
be the completion of $U(\g)$ with respect to this topology. Since each $U(\g) F_N$ is a left ideal, $\hat{U}$ is naturally a left $U(\g)$-module. We claim that $\hat{U}$ is also an algebra. We may identify $\hat{U}$ with the space of infinite sums of monomial elements of $U(\g)$ in which only finitely many terms lie outside $U(\g)F_N$ for each $N \geq 0$. Let $x = \sum_{i \in \Z_+} x^{(i)}$ and $y = \sum_{j \in \Z_+} y^{(j)}$ be two such sums. If $y^{(j)} \in U(\g) F_N$ then clearly $x^{(i)} y^{(j)} \in U(\g) F_N$. Using the commutation relations one may check the following: for each of the other finitely many $y^{(j)}$, there exists $N^{(j)}$ such that if $x^{(i)} \in U(\g) F_{N^{(j)}}$ then $x^{(i)} y^{(j)} \in U(\g) F_N$. So only finitely many $x^{(i)} y^{(j)}$ lie outside $U(\g) F_N$. The product of $x$ and $y$ is defined term-by term, and is a well-defined element of $\hat{U}$.

For $a, b \in V$, $m \in [\eps_a]$, $k \in [\eps_b]$, and $n \in \Z$, let
\begin{align}
\begin{split}
BI(a, b; m, k; n) &= \sum_{j \in \Z_+} \binom{m+\D_a-1}{j} (a_{(n+j)}b)_{m+k} \\
& \phantom{\lim} - \sum_{j \in \Z_+} (-1)^j \binom{n}{j} \left[ a_{m+n-j}b_{k+j-n} - (-1)^n b_{k-j} a_{m+j} \right],
\end{split}
\end{align}
which we may think of as an element of $\hat{U}$ because, for each $r$, all but finitely many terms lie in $U(\g)F_r$. Let $\B \subseteq \hat{U}$ be the span of the terms $\vac_n - \delta_{n, 0} 1$ for $n \in \Z$ and $BI(a, b; m, k; n)$ as $a$, $b$, $m$, and $k$ range over all their possible values, and let $S = \B \tilde{M} \subseteq \tilde{M}$. Because $Y(a, z)$, $Y^M(a, z)$, and $Y^M(b, z)$ are quantum fields, elements of $\B \tilde{M}$ are finite sums, hence $S$ is well-defined\footnote{We could, therefore, have avoided the introduction of the completion $\hat{U}$ here. However, $\hat{U}$ will appear in a later section.}.

In Lemma 2.26 of \cite{DK} it is proved by direct calculation that $[c_s, \B] \subset \B$ for all $c \in V$, $s \in [\eps_c]$. This implies that $S$ is an $\g$-submodule of $\tilde{M}$, so let $M = \tilde{M} / S$. On the quotient $M$, $Y^M(a, z)$ is a quantum field, and the vacuum and Borcherds identities are satisfied. Therefore $M$ is a $V$-module.

Lemma \ref{lithesis}, and that $\g_{>p}$ annihilates $N$, imply that $M$ has no pieces of negative degree. Therefore $M$ is actually a $(\G, H)$-twisted positive energy $V$-module. We claim now that $M_p = N$.

Let $x \in N$ and $y = (a^1)^M_{n_1} \cdots (a^s)^M_{n_s} x \in M_p$ (so $\sum n_i = 0$). By Lemma \ref{lithesis}, $y = s + x'$ where $s \in S_p$ and $x' = b^M_0 x \in N$ for some $b \in V_\G$. Therefore $\tilde{M}_p = S_p + N$. Since we want to show that this sum is direct, it suffices to show that $S_p \cap N = 0$. Our strategy for proving this is a hybrid of the strategies of \cite{DLM} and \cite{DK}.

Following \cite{DLM} we introduce a bilinear form $\left< \cdot, \cdot \right> : N^* \times \tilde{M} \rightarrow \C$.
\begin{defn} \label{pairingdef}
Let $\psi \in N^*$. If $x \in N$, set $\left< \psi, x \right> = \psi(x) \in \C$, i.e., $\left< \cdot, \cdot \right>$ restricts to the canonical form on $N^* \times N$. For each $n \in \overline{\G}$ such that $n \neq 0$ and $n \leq p$, fix an ordered basis $B_n$ of $(\lie V)_n$; their union $B$ is a basis of a subspace of $\g$ complementary to $\g_+$. Let $s \geq 2$ be an integer, and let $a^1_{n_1}, \ldots a^s_{n_s} \in B$ be ordered lexicographically, i.e., $n_1 \leq n_2 \leq \ldots \leq n_s$ and if consecutive $n_i$ are equal then the corresponding $a^i$ are in increasing order in $B_{n_i}$. Suppose also that $\sum_i n_i = 0$. Choose $c \in V$ such that $c_{n_1+n_2} = a^1_{n_1} a^2_{n_2}$ (Lemma \ref{lithesis} states that such an element exists, it may not be unique though) and define
\begin{align} \label{complicated}
\left< \psi, (a^1)^M_{n_1} \cdots (a^s)^M_{n_s} x \right>
= \left< \psi, c^M_{n_1+n_2} (a^3)^M_{n_3} \cdots (a^s)^M_{n_s} x \right>.
\end{align}
If $s = 2$ then $n_1+n_2 = 0$ and $c^M_0 x \in N$. If $s \geq 3$ then equation (\ref{complicated}) defines $\left<\cdot, \cdot\right>$ inductively. If $q \neq p$ and $y \in \tilde{M}_q$ then set $\left<\psi, y\right> = 0$. Finally we extend the definition to all elements of $M$ linearly.
\end{defn}

The bilinear form $\left<\cdot, \cdot\right>$ is well-defined because of the PBW theorem, i.e., we defined it on a basis of $\tilde{M}$. In \cite{DLM}, choices of elements $c$ were given by an explicit formula. We omit such a formula, since we do not require it.

For any $a, b \in V$ with $[\eps_a] + [\eps_b] = [0]$, and for all $\psi \in N^*$, $x \in N$ we will show that
\begin{align} \label{particular}
\left< \psi, BI(a, b; p+1+\eps_a, -(p+1+\eps_a); -1) x \right> = 0.
\end{align}
The proof of this fact is fairly involved, so we relegate it to appendix \ref{appcalc}; it uses the fact that $N$ is a $\zhu_{p, \G}(V)$-module, and it uses our particular choice of $\left< \cdot, \cdot \right>$.

In Lemma 2.27 of \cite{DK} it was proved that $\B$ is spanned by all elements of the form $BI(a, b; m_a, k; -1)$ where $a$ and $b$ range over $V$, $k$ ranges over $[\eps_b]$, and $m_a$ is a fixed number in $[\eps_a]$ (which can be chosen as we please for each $a$). From equation (\ref{particular}) and this lemma it follows that $\left<\psi, \B N \right> = 0$.

Recall the notion of a local pair. Let $U$ be a vector superspace and let $a(w)$, $b(w)$ be $\en U$-valued quantum fields. We say the pair $(a(w), b(w))$ is \emph{local} if there exists $n \in \Z_+$ such that
\[
(z-w)^n [a(z), b(w)] = 0.
\]
If $a(w) = \sum_{n \in [\ga_a]} a_{(n)} w^{-n-1}$ and $b(w) = \sum_{n \in [\ga_b]} b_{(n)} w^{-n-1}$ then locality of the pair $(a(w), b(w))$ is equivalent to the existence of a finite collection of quantum fields $c^j(w)$, $j = 0, 1, \ldots N$, such that
\[
[a_{(m)}, b_{(k)}] = \sum_{j \in \Z_+} \binom{m}{j} c^j_{(m+k-j)}
\]
(see \cite{Kac} and \cite{DK}). In particular, the quantum fields $Y^M(a, w)$ are pairwise local because of equation (\ref{conbrak}).

\begin{lemma}[Dong's Lemma]
Let $a(w)$, $b(w)$ and $c(w)$ be pairwise local quantum fields. Then $a(w)$ and the $\G$-twisted $n^\text{th}$ product\footnote{See Remark 2.20 of \cite{DK} for an explanation of $\G$-twisted $n^\text{th}$ products. In particular, of the difference between this definition and equation (\ref{nthproddef}).}
\[
b(w)_{(n, \G)}c(w) = \res_z (zw^{-1})^{m+\D_b-1} \left[ b(z) c(w) i_{z, w}
- p(a, b) c(w) b(z) i_{w, z} \right](z-w)^n
\]
also form a local pair.
\end{lemma}
See \cite{Kac} for a proof.

Let\footnote{We are abusing notation here by using the same symbol $BI$ to represent the expression with all instances of $a_n$ replaced by $a^M_n$.}
\begin{align*}
BI(a, b; m; n; w)
= {} & \sum_{k \in \eps_b} w^{-m-\D_a-k-\D_b+n-1} BI(a, b; m, k; n) \\
= {} & \sum_{j \in \Z_+} \binom{m+\D_a-1}{j} Y^M(a_{(n+j)}b, w) w^{-j}
- Y^M(a, z)_{(n, \G)}Y^M(b, w).
\end{align*}
Dong's lemma, together with the finiteness of the sum in the second line, implies that $BI(a, b; m; n; w)$ is a quantum field which forms a local pair with any quantum field $Y^M(c, w)$.

The following is Lemma 2.23 of \cite{DK}.
\begin{lemma}[Uniqueness Lemma] \label{uniqueness}
Let $U$ be a vector superspace, $\mathcal{F} = \{a^i(w) | i \in I\}$ a collection of $\en U$-valued quantum fields that are pairwise local, and $b(w)$ an $\en U$-valued quantum field which forms a local pair with each element of $\mathcal{F}$. Suppose $W$ is a \emph{generating subspace} of $U$, meaning that $U$ is spanned by vectors of the form $a^{i_1}_{(n_1)} \cdots a^{i_s}_{(n_s)} x$ where $x \in W$. If $b(w) W = 0$ then $b(w) U = 0$.
\end{lemma}

We apply the lemma with $U = \tilde{M}$, $W = N$, $\mathcal{F} = \{Y^M(a, w) | a \in V\}$, and $b(w) = BI(a, b; m; n; w)$, also let $\psi \in N^*$. We saw above that $\left<\psi, \B N\right> = 0$, Lemma \ref{uniqueness} implies that
\begin{align} \label{formiden}
\left< \psi, \B \tilde{M} \right> = 0
\end{align}
(the presence of $\psi$ does not affect the proof of the lemma).

Since $\left< \cdot, \cdot \right>$ restricts to the canonical pairing of $N^*$ with $N$, and (\ref{formiden}) holds for all $\psi \in N^*$, we obtain $N \cap S_p = 0$.

\begin{defn}[Induction functor $M^p$]
If $N \in \rep(\zhu_{p, \G}(V))$, we define $M^p(N) = M \in \perep(V)$ as constructed above. If $g : N \rightarrow N'$ is a homomorphism of $\zhu_{p, \G}(V)$-modules then we define $M^p(g) : M^p(N) \rightarrow M^p(N')$ as follows: $g$ is also a homomorphism of $\g_+$-modules. Let $\tilde{M} = \ind_{\g_+}^\g N$ and $\tilde{M}' = \ind_{\g_+}^\g N'$, $\ind_{\g_+}^\g$ is a Lie algebra induction functor, so $g$ lifts to a unique homomorphism $\tilde{g} : \tilde{M} \rightarrow \tilde{M}'$ of $\g$-modules such that $\tilde{g}|_N : N \rightarrow N'$ coincides with $g$. Now
\begin{align*}
\tilde{g}(S) = \tilde{g}(\B \tilde{M}) = \B \tilde{g}(\tilde{M}) \subseteq \B \tilde{M}' = S'
\end{align*}
so there is an induced map $g : M \rightarrow M'$ and we put $M^p(g) = g$. $M^p$ is a functor from $\rep(\zhu_{p, \G}(V))$ to $\perep(V)$.
\end{defn}

Let $I \subset M$ be the maximal graded $\g$-submodule of $M$ whose intersection with $N$ is $0$ and let $L = \tilde{M} / I$. Because $S$ is a graded submodule of $\tilde{M}$ meeting $N$ trivially we know $I$ exists and contains $S$, which in turn implies that $L$ is a $V$-module. The quotient map $\tilde{M}/S \twoheadrightarrow \tilde{M}/I$ is a homomorphism of $V$-modules. Consequently $L$ is a $(\G, H)$-twisted positive energy $V$-module with $L_p = N$, and such that every graded $V$-submodule of $L$ intersects $L_p$ non-trivially.

\begin{defn}
For any $N \in \rep(\zhu_{p, \G}(V))$, let $L^p(N) = L \in \perep(V)$. A homomorphism $g : N \rightarrow N'$ extends to a homomorphism $\tilde{g} : \tilde{M} \rightarrow \tilde{M}'$. Clearly $\tilde{g}(I) \subseteq I'$, we let $L^p(g)$ be the induced map $g : L^p(N) \rightarrow L^p(N')$. $L^p$ is a functor from $\rep(\zhu_{p, \G}(V))$ to $\perep(V)$.
\end{defn}


\subsection{\texorpdfstring{Some properties of the functors $\Omega_p$, $M^p$, and $L^p$}{Properties of these Functors}} \label{funcprop}

\begin{defn}
Let $M \in \perep(V)$. We say $M$ is \emph{normalized} if $M_0 \neq 0$. We say $M$ is \emph{$p$-irreducible}\footnote{In \cite{DK} $0$-irreducible $V$-modules are called \emph{almost irreducible}.} if
\begin{itemize}
\item $M$ is generated over $V$ by $M_p$, i.e., $V M_p = M$ in light of Lemma \ref{lithesis}.

\item Every nonzero graded $V$-submodule of $M$ has nonzero intersection with $M_p$.
\end{itemize}
\end{defn}
If $N \in \rep(\zhu_{p, \G}(V))$ then $L^p(N)$ is $p$-irreducible. If $M \in \perep(V)$ is irreducible then it is $p$-irreducible for each $p$ such that $M_p \neq 0$.

\begin{defn}
Let $\Phi_p : \rep(\zhu_{p-1, \G}(V)) \rightarrow \rep(\zhu_{p, \G}(V))$ be the functor sending $N \in \rep(\zhu_{p-1, \G}(V))$ to the same vector space $N$ with $\zhu_{p, \G}(V)$ acting by $a x = \phi_p(a) x$. We say $N \in \rep(\zhu_{p, \G}(V))$ is \emph{$p$-founded} if it is \emph{not} in the image of $\Phi_p$, i.e., the action of $\zhu_{p, \G}(V)$ on $N$ does not factor to an action of $\zhu_{p-1, \G}(V)$.
\end{defn}

\begin{thm} \label{fullequiv}{\ }
\begin{itemize}
\item For all $N \in \rep(\zhu_{p, \G}(V))$,
\begin{align*}
[\Omega_p \circ M^p](N) \cong [\Omega_p \circ L^p](N) \cong N.
\end{align*}
Indeed each of these compositions of functors is equivalent to the identity functor on $\rep(\zhu_{p, \G}(V))$.

\item Let $M \in \perep(V)$ be $p$-irreducible. Then $[L^p \circ \Omega_p](M) = M$.

\item $\Omega_p$ and $L^p$ restrict to inverse equivalences between the full subcategory of $\perep(V)$ of $p$-irreducible $V$-modules, and $\rep(\zhu_{p, \G}(V))$.

\item $\Omega_p$ and $L^p$ further restrict to inverse equivalences between the full subcategory of $\perep(V)$ of normalized $p$-irreducible $V$-modules, and the full subcategory of $\rep(\zhu_{p, \G}(V))$ of $p$-founded $\zhu_{p, \G}(V)$-modules.
\end{itemize}
\end{thm}

\begin{proof}
The first claim follows from the results of the last section.

For the second claim, let $N = \Omega_p(M)$, $M' = M^p(N)$, and $L = L^p(N)$. Because $M$ is generated over $V$ by $N$, we have that $M$ is a quotient of $M'$, so let $M = M' / J$. The graded submodule $J$ must have zero intersection with $M'_p$ because $M_p \cong M'_p \cong N$. It must be maximal with this property because, by assumption, all nonzero graded submodules of $M$ have nonzero intersection with $M_p$. Thus $M = M' / J = L$.

The third claim follows from the first two.

For the fourth claim, let $N \in \rep(\zhu_{p, \G}(V))$ and suppose $M = M^p(N)$ is not normalized, i.e., that $M_0 = 0$. We lower the degree of every vector in $M$ by $1$, and apply the functor $\Omega_{p-1}$, to obtain $N$ as a $\zhu_{p-1, \G}(V)$-module. This action of $\zhu_{p-1, \G}(V)$ on $N$ is factored from the $\zhu_{p, \G}(V)$-action via $\phi_p$. Therefore if $N$ is $p$-founded then $M$ is normalized.

On the other hand if $N$ is not $p$-founded then $M_0 = 0$; this is implied by the following claim: Let $D_+ : \perep(V) \rightarrow \perep(V)$ be the functor that raises the degree of every vector in a module by $1$. Then
\begin{align*}
\begin{split}
\xymatrix{
\rep(\zhu_{p-1, \G}(V)) \ar@{->}[d]_{L^{p-1}} \ar@{->}[r]^{\Phi_p} & \rep(\zhu_{p, \G}(V)) \ar@{->}[d]_{L^p} \\
\perep(V) \ar@{->}[r]_{D_+} & \perep(V) \\
}
\end{split}
\end{align*}
commutes.

Let $N \in \rep(\zhu_{p-1, \G}(V))$. Recall Lemma \ref{liezhu}, and in particular the commutative diagram (\ref{liecomm}), it implies that $N$ acquires the same $\g_0$-module structure via $\zhu_{p-1, \G}(V)$ and via $\zhu_{p, \G}(V)$ (recall that $\g = \lie V$).

In the construction of $L^{p-1}(N)$ we induce from $N$ as a module over $\g_+ = \g_0 + \g_{>p-1}$ to get the $\g$-module $\tilde{M}$. In the construction of $L^p(N)$ we induce from $N$ as a module over $\g'_{+} = \g_0 + \g_{>p}$ to get the $\g$-module $\tilde{M}'$. In each case there is a choice of grading, and these are different by our conventions; this is taken care of by $D_+$.

By the PBW theorem, both $\tilde{M}$ and $\tilde{M}'$ are spanned by monomials
\begin{align*}
(a^1)^M_{n_1} \cdots (a^s)^M_{n_s} x,
\end{align*}
where $x \in N$, $a^i \in V$ and $n_1 \leq n_2 \leq \ldots \leq n_s$. For $\tilde{M}$, $n_s \leq p-1$ and for $\tilde{M}'$, $n_s \leq p$. There is a natural homomorphism of $\g$-modules $\tilde{M}' \twoheadrightarrow \tilde{M}$ whose kernel is the span of monomials with $p-1 < n_s \leq p$. In other words $\tilde{M} = \tilde{M}' / U(\g)K_p N$, where $K_p = \oplus_{p-1 < j \leq p} \g_j$.

We then obtain our $V$-modules $L^{p-1}(N)$ and $L^p(N)$ as quotients of $\tilde{M}$ and $\tilde{M}'$ by their respective maximal $\g$-submodules having zero intersection with $N$. If we show that the $V$-submodule of $L^{p-1}(N)$ generated by $K_p N$ has zero intersection with $N$, then we have $L^{p-1}(N) \cong L^p(N)$. By Lemma \ref{lithesis}, this intersection is $\oplus_{p-1 < j \leq p} \g_{-j} \g_j N \cap N$. Hence it suffices to show that $a_{-j} b_j x = 0$ for all $a, b \in V$ such that $[\eps_a] + [\eps_b] = [0]$, $x \in N$, and $p-1 < j \leq p$ (here we write $a_j$ in place of the more cumbersome $a^{L^{p-1}(N)}_j$).

Lemma \ref{akbk} implies that $a_{-j} b_j x$ is a linear combination of terms of the form $(a_{[n]}b)_0 x$ where
\begin{align*}
n \leq -p-1-j+\eps_a < -2p+\eps_a \leq -2(p-1)-2+\chi(a, b).
\end{align*}
This inequality implies that $a_{-j} b_j x = 0$ because $N$ is a $\zhu_{p-1, \G}(V)$-module.
\end{proof}

\begin{defn}
The $\G/\Z$-graded vertex algebra $V$ is \emph{$(\G, H)$-rational} if
\begin{itemize}
\item $\perep(V)$ has finitely many irreducible elements up to isomorphism and degree-shifts.

\item The graded subspaces $M_n$ of an irreducible $M \in \perep(V)$ are each finite-dimensional.

\item Any element of $\perep(V)$ is a direct sum of irreducible elements.
\end{itemize}
\end{defn}

\begin{prop}
If $V$ is $(\G, H)$-rational then each $\zhu_{p, \G}(V)$ is a finite dimensional semisimple algebra.
\end{prop}

\begin{proof}
Let $N \in \rep(\zhu_{p, \G}(V))$ and put $M = L^p(N)$. Then $M$ decomposes into a direct sum of irreducible $(\G, H)$-twisted positive energy $V$-modules and so $N = \Omega_p(M)$ is a direct sum of irreducible $\zhu_{p, \G}(V)$-modules by Proposition \ref{smallirred}. Thus $\zhu_{p, \G}(V)$ is semisimple.

In particular $\zhu_{p, \G}(V)$ itself decomposes into a direct sum of irreducible modules (left ideals). We can decompose the unit element as a finite sum $\vac = e_1 + \cdots + e_n$ where each $e_i$ lies in a distinct ideal. These are precisely the nonzero ideals in the sum. Because of rationality each ideal has finite dimension and thus $\zhu_{p, \G}(V)$ itself has finite dimension.
\end{proof}

We can prove a converse to this result if we assume our vertex algebra $V$ is a vertex operator algebra with nonzero central charge $c$, and that elements of $V$ have only non-negative conformal weights.
\begin{lemma} \label{infinite}
Let $V$ be a $\G/\Z$-graded vertex operator algebra with energy-momentum field $L(z) = Y(\omega, z)$ and central charge $c \in \C \backslash \{0\}$. Let $M \in \perep(V)$. Then
\begin{itemize}
\item $M$ is not bounded (we say a graded $V$-module $M$ is \emph{bounded} if $M_n$ is nonzero for only finitely many $n$).

\item If $M$ is irreducible then $L^M_{-1} : M_n \rightarrow M_{n+1}$ is injective for $n \gg 0$.

\item If $M$ is irreducible then $M_n \neq 0$ for $n \gg 0$.
\end{itemize}
\end{lemma}

\begin{proof}
Suppose $M$ is bounded, then there exists $N > 0$ such that $L^M_n M = 0$ whenever $|n| \geq N$. For any $x \in M$ we have
\begin{align*}
(2N+1) L^M_1 x &= [L^M_{N+1}, L^M_{-N}] x = 0 \\
\text{and} \quad
(2N+1) L^M_{-1} x &= [L^M_N, L^M_{-N-1}] x = 0, \\
\text{hence} \quad
2 L^M_0 x &= [L^M_1, L^M_{-1}] x = 0.
\end{align*}
But also
\begin{align*}
0 = [L^M_N, L^M_{-N}] x = 2N L^M_0 x + \frac{N^3-N}{12} c x.
\end{align*}
So $c = 0$, which is a contradiction.

Now assume $M$ is irreducible. Let $x \in M$ be of homogeneous degree, suppose $L^M_{-1} x = 0$, and let $a \in V$ be of homogeneous conformal weight. Either $Y^M(a, z) x = 0$, or there is some $a^M_{(k)} x \neq 0$. Suppose the latter case holds, and let $k \in [\ga_a]$ be maximal with this property. Then
\begin{align*}
-(k+1) a^M_{(k)} x = [L^M_{-1}, a^M_{(k+1)}] x = 0,
\end{align*}
so we must have $k = -1$. If $[\ga_a] \neq [0]$ we have a contradiction, so $Y^M(a, z) x = 0$. If $[\ga_a] = [0]$ then we have $a^M_{(n)} x = 0$ for $n \geq 0$. In either case we have $a^M_m x \neq 0$ only if $m < -\D_a + 1 \leq 1$.

Therefore the elements of the $V$-submodule $V x \subseteq M$ have degrees at least $\deg x-1$. Let $d$ be the lowest degree that occurs among the elements of $M$. If $\deg x > d+1$ then $V x$ is a proper submodule of $M$, which contradicts irreducibility. Thus there exists $N = d+1 > 0$ such that if $\deg x \geq N$, then $L^M_{-1} x \neq 0$.

The final statement follows from the first two. $M$ is unbounded, i.e., there is a nonzero $M_n$ for some $n \geq N$. But then $L^M_{-1}$ is injective on $M_j$ for $j \geq n$. Hence $M_n, M_{n+1}, M_{n+2}, \ldots \neq 0$.
\end{proof}

\begin{thm} \label{rati}
Let $V$ be a $\G/\Z$-graded vertex operator algebra with energy-momentum field $L(z) = Y(\omega, z)$ with central charge $c \neq 0$. If $\zhu_{p, \G}(V)$ is a finite dimensional semisimple algebra for each $p \in \R_+$ then $V$ is $(\G, H)$-rational.
\end{thm}

For this theorem and its proof we use that all the Zhu algebras $\zhu_{p, \G}(V)$ (even for $p \notin \Z_+$) are finite dimensional and semisimple. The general definition of the Zhu algebras is given in appendix \ref{pnotinZ}.

\begin{proof}
The normalized irreducible $(\G, H)$-twisted positive energy $V$-modules are in bijection with the irreducible $\zhu_{0, \G}(V)$-modules. Therefore, up to isomorphism and degree-shifts, there are finitely many irreducible objects in $\perep(V)$. For each such normalized irreducible $M$, $M_p$ is an irreducible $\zhu_{p, \G}(V)$-module, hence finite dimensional. Let $M \in \perep(V)$ be one such irreducible object. By the argument from Remark \ref{inequivrem}, $L^M_0|_{M_n} = \la + n$ for all $n \in \R_+$, for some $\la \in \C$. By Lemma \ref{infinite}, $M$ has some graded piece $M_n$ such that it and all higher graded pieces of $M$ are nonzero. Therefore there exists $K \in \Z_+$ with the following property: For all irreducible $M \in \perep(V)$ and $n \in \R_+$ such that $\Real(L^M_0|_{M_n}) > K$, we have $M_n \neq 0$.

Let $M \in \perep(V)$, suppose $M$ is normalized, and let $N = \Omega_0(M)$. Since $\zhu_{0, \G}(V)$ is semisimple, we have the direct sum of submodules $N = N' \oplus N''$ where $N'$ is irreducible. We have $L^M_0|_{N'} = \lambda$ a scalar.

Let $M' = V N' \subseteq M$, and let $W$ be the irreducible quotient of $M'$. Let $K$ be as above and let $p > K - \lambda$. We have
\[
W \cong L^p(W_p) = M^p(W_p) / J,
\]
where $J$ is the unique maximal ideal such that $J_p = 0$. Either $J = 0$ or the irreducible quotient $\overline{J}$ of $J$ (which satisfies $\overline{J}_p = 0$) is zero because of our choice of $p$. Either way, $J = 0$ and so $M^p(W_p) \cong W \cong L^p(W_p)$.

The canonical surjection $M' \twoheadrightarrow W$ induces $M'_p \twoheadrightarrow W_p$. Because $\zhu_{p, \G}(V)$ is semisimple, we obtain an inclusion $W_p \hookrightarrow M'_p$, so we regard $W_p \subseteq M_p$ now. There are homomorphisms of $V$-modules
\[
M^p(W_p) \rightarrow V W_p \rightarrow L^p(W_p),
\]
these must each be isomorphisms. Therefore $V W_p \subseteq M'$ is an irreducible $V$-module (isomorphic to $W$).

By Lemma \ref{lithesis}, $V W_p \cap M_0 \subseteq N'$. The submodule $\ker(M' \twoheadrightarrow W) \subseteq M'$ cannot contain $N'$, for then it would equal all of $M'$. Hence $V W_p$ has nonzero intersection with $N'$. But $N'$ is irreducible, so $VW_p \cap N' = N'$. Hence $M' = V W_p$, an irreducible submodule of $M$ with $M'_0 = N'$.

We apply this argument to each irreducible summand of $N$ to obtain a sum (which is obviously direct) of irreducible $V$-modules $M^1 \oplus \cdots \oplus M^n \subseteq M$. If this inclusion is an equality then we are done. If not, then we let $r \in \R_+$ be minimal such that $(M / (M^1 \oplus \cdots \oplus M^n))_r \neq 0$. Because $\zhu_{r, \G}(V)$ is semisimple, we have $\Omega_r(M) = (M^1 \oplus \cdots \oplus M^n)_r \oplus N^{(r)}$ for some $\zhu_{r, \G}(V)$-module $N^{(r)}$. We repeat the arguments above on $N^{(r)}$, obtaining further irreducible summands. Ultimately, we can write $M$ as a (direct) sum of irreducible $V$-modules in this way.
\end{proof}

\begin{rmrk}
Using the maps $\phi_p$ it is possible to use the same proof to establish this theorem if $\zhu_{p, \G}(V)$ is semisimple and finite-dimensional only for $p \gg 0$.
\end{rmrk}


\section{Computation of Zhu algebras} \label{examples}

\subsection{\texorpdfstring{Alternative constructions of $\zhu_{p, \G}(V)$}{Alternative Constructions of the Zhu Algebra}}

Let $V$ be a $\G/\Z$-graded vertex algebra with associated Lie superalgebra $\g = \lie V$. In section \ref{inducfunct} we introduced the completed universal enveloping algebra $\hat{U}$, and its subspace $\B$. Let $\hat{\B}$ denote the closure of $\B$ in $\hat{U}$.

The important Lemma 2.26 of \cite{DK} states that $[c_s, \B] \subseteq \B$ for all $c \in V$, $s \in [\eps_c]$. It follows that $\hat{U} \hat{\B}$ is a (graded) $2$-sided ideal in $\hat{U}$. Therefore we may define
\[
U_{\G}(V) = \hat{U} / \hat{U} \hat{\B},
\]
a graded unital associative algebra.

The degree $0$ piece $(U_\G(V))_0$ is a subalgebra; its subspace $(U_\G(V) (U_\G(V)_{>p}))_0$ is a $2$-sided ideal. Therefore we have another unital associative algebra\footnote{In both cases the unit element $1$ might coincide with $0$ though.}
\[
W_{p, \G}(V) = \frac{(U_{\G}(V))_0}{(U_{\G}(V) (U_{\G}(V)_{>p}))_0}.
\]
We also define
\[
Q_{p, \G}(V) = \frac{U_{\G}(V)}{U_{\G}(V) (U_{\G}(V)_{>p})}.
\]
Since $U_{\G}(V) (U_{\G}(V)_{>p})$ is a left ideal in $U_{\G}(V)$ but not a $2$-sided ideal, $Q_{p, \G}(V)$ is a left $U_\G(V)$-module but not an algebra.

Let us mimic the construction of section \ref{inducfunct} on the $\g_0$-module $N = U(\g_0)$. We obtain the $U(\g)$-module $\tilde{M} = U(\g) / U(\g)(U(\g)_{>p})$, and then the quotient $M = \tilde{M} / \B \tilde{M}$ which is a $(\G, H)$-twisted positive energy $V$-module. One may check that $M \cong Q_{p, \G}(V)$ (having degree-shifted the latter module to put the unit element in degree $p$), thus $M_p \cong W_{p, \G}(V)$. We claim that $W_{p, \G}(V) \cong \zhu_{p, \G}(V)$.

Let $M$ be a $(\G, H)$-twisted positive energy $V$-module. Since $M$ is a $U(\g)$-module such that $a^M_n x = 0$ for $n \gg 0$, it makes sense to regard $M$ as a $\hat{U}$-module. Because $M$ is a $V$-module, it is a $U_{\G}(V)$-module. Each graded piece of $M$ is a $U_{\G}(V)_0$-module. Furthermore $(U_{\G}(V)U_{\G}(V)_{>p})_0$ annihilates $M_p$, so $M_p$ is a $W_{p, \G}(V)$-module. In particular if we induce the adjoint representation of $\zhu_{p, \G}(V)$ to the $(\G, H)$-twisted positive energy $V$-module $M$, such that $M_p = \zhu_{p, \G}(V)$, we see that $\zhu_{p, \G}(V)$ is a $W_{p, \G}(V)$-module. This provides us with a homomorphism $W_{p, \G}(V) \rightarrow \zhu_{p, \G}(V)$ taking $[a_0] \mapsto [a]$.

On the other hand $Q_{p, \G}(V)$ is a $(\G, H)$-twisted positive energy $V$-module and its degree $p$ piece is $W_{p, \G}(V)$, which is therefore a $\zhu_{p, \G}(V)$-module. We obtain a homomorphism $\zhu_{p, \G}(V) \rightarrow W_{p, \G}(V)$ taking $[a] \mapsto [a_0]$. Thus $W_{p, \G}(V) \cong \zhu_{p, \G}(V)$.

It is clear that $W_{p, \G}(V)$ is associative and unital. It is straightforward to see that there is a restriction functor $M \mapsto M_p$ from $\perep(V)$ to $\rep(W_{p, \G}(V))$. However it is not clear from its definition that $W_{p, \G}(V)$ has an associated induction functor with the desired properties. This is the advantage that the construction $\zhu_{p, \G}(V)$ has over $W_{p, \G}(V)$.

In section \ref{maps} we noted that the set of all Zhu algebras of a fixed vertex algebra $V$, with the maps $\phi_p$, forms a directed system. The set of algebras $W_{p, \G}(V)$, with the natural projections, forms an isomorphic directed system. The inverse limit of this directed system is equal to the completion of $U_\G(V)_0$ with respect to the system $(U_\G(V) U_\G(V)_{>p})_0$ of neighborhoods of $0$. Indeed
\begin{align*}
\varprojlim W_{p, \G}(V) \cong U_{\G}(V)_0.
\end{align*}

\subsection{A simplified construction for universal enveloping vertex algebras}

In this section we discuss Lie conformal algebras and their universal enveloping vertex algebras. The facts we use are laid out in section 1.7 of \cite{DK}, more details can be found in \cite{Kac}. For simplicity we deal only with the untwisted case here.

\begin{defn}
A \emph{Lie conformal (super)algebra} $(R, T, [\cdot_\la \cdot])$ is a $\C[T]$-module $R$, endowed with a \emph{$\la$-bracket}, $R \otimes R \rightarrow R \otimes \C[\la]$ (denoted $a \otimes b \mapsto [a_{\la}b]$) for which the following axioms hold:
\begin{itemize}
\item Sesquilinearity axiom, $[(Ta)_{\la}b] = -\la [a_{\la}b]$ and $[a_{\la}(Tb)] = (T + \la)[a_{\la}b]$.

\item Skew-commutativity axiom, $[b_{\la}a] = -p(a, b)[a_{-T-\la}b]$.

\item Jacobi Identity, $[a_{\la}[b_{\mu}c]] = [[a_{\la}b]_{\la + \mu}c] + p(a, b)[b_{\mu}[a_{\la}c]]$.
\end{itemize}
We write $[a_{\la}b] = \sum_{j \in \Z_+} \la^{(j)} a_{(j)}b$ where $a_{(j)}b \in R$. A homomorphism of Lie conformal algebras is a homomorphism of $\C[T]$-modules that sends $\la$-brackets to $\la$-brackets.
\end{defn}

\begin{exmp}
The \emph{Virasoro} Lie conformal algebra is $\vir = \C[T] L \oplus \C C$ with $TC = 0$. It suffices to define the $\la$-brackets between $L$ and $C$ and then all further $\la$-brackets are determined by the sesquilinearity axiom; $[L_{\la}L] = (T + 2\la)L + \la^3 C / 12$ and $C$ is central. We have $L_{(0)}L = TL$, $L_{(1)}L = 2L$, $L_{(3)}L = C / 2$ and all other $L_{(n)}L$ vanish.
\end{exmp}

\begin{exmp}
Let $\g$ be a finite-dimensional simple Lie algebra with invariant symmetric bilinear form $(\cdot, \cdot)$. The \emph{Current} Lie conformal algebra $\cur(\g)$ is $\cur(\g) = (\C[T] \otimes \g) \oplus \C K$ with $TK = 0$ (the elements $a \in \g \subseteq \cur(\g)$ are referred to as currents). The $\la$-brackets between currents are $[a_\la b] = [a, b] + \la (a, b) K$, the element $K$ is central. We have $a_{(0)}b = [a, b]$ and $a_{(1)}b = (a, b) K$.
\end{exmp}

Any vertex algebra can be given the structure of a Lie conformal algebra by defining the $\la$-bracket $[a_{\la}b] = \sum_{j \in \Z_+} \la^{(j)} a_{(j)}b$. One should keep in mind the following analogy: Vertex algebras are to Lie conformal algebras as associative algebras are to Lie algebras.

For a Lie conformal algebra $(R, T, [\cdot_{\lambda} \cdot])$, an enveloping vertex algebra is a pair $(U, \phi)$ where $U$ is a vertex algebra and $\phi : R \rightarrow U$ is a homomorphism of Lie conformal algebras. The universal enveloping vertex algebra $(V(R), \psi)$ is an enveloping vertex algebra such that for any enveloping vertex algebra $(U, \phi)$ there is a unique vertex algebra homomorphism $\pi : V(R) \rightarrow U$ such that the following diagram commutes
\begin{align*}
\xymatrix{
R \ar@{->}[d]_{\psi} \ar@{->}[dr]^{\phi} \\
V(R) \ar@{->}[r]_{\pi} & U \\
}
\end{align*}

As a vector space $V(R) = U(\lie R) / U(\lie R) (\lie R)_-$, where
\begin{align*}
(\lie R)_- &= \left<a_{(n)} | a \in R, n \geq 0\right> \\
\text{and} \quad
(\lie R)_+ &= \left<a_{(n)} | a \in R, n < 0\right>,
\end{align*}
and the map $R \rightarrow V(R)$ is $a \mapsto a_{(-1)}$.

If a Lie conformal algebra $R$ contains a copy $\C[T]L \oplus \C C \subseteq R$ of the Virasoro Lie conformal algebra, if $L_{(0)}a = Ta$ for all $a \in R$, and if $L_{(1)}$ is diagonalizable on $R$, then the image of $L$ in $V(R)$ is a Virasoro element. We assume our $R$ contains such an element $L$ for the remainder of this section. For $a \in R$ an eigenvector of $L_{(1)}$, let $\D_a$ be its eigenvalue.

Recall the definition of $\lie V$ following Lemma \ref{THann}. It turns out that it uses only that $V$ is a Lie conformal algebra. Hence we may define $\lie R$ for an arbitrary Lie conformal algebra by replacing all occurrences of $V$ in the definition with $R$. The conformal weight grading $a_n = a_{(n+\D_a-1)}$ is defined as usual.

We have that $V(\vir) / (C = c\vac) = \vir^c$ is the Virasoro vertex algebra of level $c$, and that $\lie \vir$ is the Virasoro Lie algebra (which is also given the symbol $\vir$). Similarly $V(\cur{\g}) / (K = k\vac) = V^k(\g)$, and $\lie \cur{\g} = \hat{\g}$, the affine Kac-Moody algebra.

\begin{lemma} \label{restrlieR}
There is a natural bijection between the set of all $V(R)$-modules and the set of restricted $\lie R$-module.
\end{lemma}

Recall that a \emph{restricted} $\lie R$-module is a $\lie R$-module $M$ such that for each $a \in R$ and $v \in M$ we have $a_{(n)} v = 0$ for $n \gg 0$.

\begin{proof}
Since $V(R) = U(\lie R) / U(\lie R) (\lie R)_-$, any $V(R)$-module $M$ is naturally a $\lie R$-module. Because each $Y^M(a, z)$ is a quantum field, $M$ is a \emph{restricted} $\lie R$-module.

Conversely, let $M$ be a restricted $\lie R$-module. Then we have a collection of $\en M$-valued quantum fields
\begin{align*}
\mathcal{R} = \{Y^M(a, z) = \sum_{n \in \Z} (at^n) z^{-n-1}\}_{a \in R}.
\end{align*}
Let $L$ denote the subalgebra of $\en M$ generated by $1$ and the operators $at^n|_M$. The algebra $L$ carries a derivation $T$, defined by $T(at^m) = -\partial_t(at^m)$, and extended naturally to all of $L$. The triple $(L, T, \mathcal{R})$ is a regular formal distribution Lie algebra (see \cite{Kac}), so it has an associated vertex algebra $V_1$. There is a representation of $V_1$ on $M$, extending the action of $R$.

There is a Lie conformal algebra homomorphism $R \to V_1$. Hence there is a vertex algebra homomorphism $\pi$ such that
\begin{align*}
R \overset{\psi}{\to} V(R) \overset{\pi}{\rightarrow} V_1.
\end{align*} 
Therefore there is a representation of $V(R)$ in $M$, which extends the representation of $\lie R$ in $M$.
\end{proof}

\begin{cor}
The level $p$ Zhu algebra of $V(R)$ is isomorphic to
\begin{align*}
Z_p(V(R)) = \frac{U(\lie R)_0}{(U(\lie R) U(\lie R)_{>p})_0}.
\end{align*}
\end{cor}

\begin{proof}
The idea is the same as for the construction of $W_{p, \G}(V)$. We apply $M^p$ to the left adjoint representation of $\zhu_p(V(R))$, to get a positive energy $V(R)$-module $M$ such that $M_p = \zhu_p(V(R))$. By Lemma \ref{restrlieR}, $M$ is a restricted $\lie R$-module, and thus a $U(\lie R)$-module.

Now, $M_p = \zhu_p(V(R))$ is a $U(\lie R)_0$-module on which $(U(\lie R) U(\lie R)_{>p})_0$ acts by $0$. Therefore $\zhu_p(V(R))$ is a $Z_p(V(R))$-module, and there is an algebra homomorphism $Z_p(V(R)) \rightarrow \zhu_p(V(R))$.

On the other hand, $M = U(\lie R) / U(\lie R) U(\lie R)_{>p}$ is a restricted $\lie R$-module and therefore a $V(R)$-module. If we shift the grading on $M$ to put $U(\lie R)_0$ in degree $p$, then $M$ is a positive energy $V(R)$-module. Since $M_p = Z_p(V(R))$, we have that $Z_p(V(R))$ is a $\zhu_p(V(R))$-module. There is an algebra homomorphism $\zhu_p(V(R)) \rightarrow Z_p(V(R))$.

The homomorphisms we constructed are mutually inverse, hence they are isomorphisms.
\end{proof}

\begin{exmp}
Let $U^k(\hat{\g}) = U(\hat{\g}) / (K = k)$, then
\begin{align*}
\zhu_p(V^k(\g)) \cong U^k(\hat{\g})_0 / (U^k(\hat{\g}) U^k(\hat{\g})_{>p})_0.
\end{align*}
When $p = 0$ we can simplify the right hand side to $U^k((\hat{\g})_0) \cong U(\g)$. The isomorphism $\zhu(V^k(\g)) \cong U(\g)$ was proved in \cite{FZ} by a different method.

In the same way we obtain $\zhu_0(\vir^c) \cong \C[x]$. Again this was proved in \cite{FZ}.

From the PBW theorem, we see that $\zhu_1(\vir^c)$ is spanned by the monomials $L_0^k$ and $L_{-1} L_0^k L_1$ for $k \in \Z_+$. Indeed, this algebra is generated by the elements $L = L_0$ and $A = L_{-1} L_1$. We have
\begin{align*}
AL = L_{-1} L_1 L_0 = L_{-1} L_0 L_1 + L_{-1} L_1 = L_0 L_{-1} L_1 - L_{-1} L_1 + L_{-1} L_1 = L_0 L_{-1} L_1 = LA.
\end{align*}
Hence $\zhu_1(\vir^c)$ is commutative, hence a quotient of $\C[A, L]$. We also have
\begin{align*}
A^2 = L_{-1} L_1 L_{-1} L_1 = L_{-1} L_{-1} L_1 L_1 + 2 L_{-1} L_0 L_1 = 2 L_0 L_{-1} L_1 + 2 L_{-1} L_1 = 2LA + 2A.
\end{align*}
Therefore
\[
\zhu_1(\vir^c) \cong \C[A, L] / \left(A^2 - 2LA - 2A\right)
\]
The higher level Zhu algebras of $\vir^c$ may be computed explicitly in a similar way, but are not commutative in general.
\end{exmp}

We may introduce a topology on $(U(\lie R))_0$ by declaring $\{(U(\lie R)U(\lie R)_{>r})_0\}_{r \geq 0}$ to be a fundamental system of neighborhoods of $0$. We have
\begin{align*}
\varprojlim \zhu_p(V(R)) \cong \widehat{(U(\lie R))_0},
\end{align*}
the completion with respect to this topology.

\subsection{Results for rational vertex algebras}

In \cite{FZ}, Frenkel and Zhu noted the following: If $V$ is a vertex algebra and $I \subseteq V$ is an ideal, then
\[
\zhu(V / I) \cong \zhu(V) / \zhu(I),
\]
where $\zhu(I)$ denotes the image of $I$ in $\zhu(V)$ under the canonical map $V \rightarrow \zhu(V)$. The same is true for $\zhu_p(V/I)$, and the proof is the same.

Consider the simple vertex algebra $V_k(\g) = V^k(\g) / I$, where $k \in \Z_+$, and $I$ is the unique maximal ideal. It is known that $I$ is generated (under left multiplication) by the element $(e_\theta)^{k+1}_{(-1)}\vac$, where $e_\theta$ is the highest root vector in $\g$. Also, it is known that $V_k(\g)$ is rational. It was shown in \cite{FZ} that $\zhu_0(I) = (e_{\theta}^{k+1}) \subseteq U(\g) = \zhu_0(V^k(\g))$, and so
\[
\zhu_0(V_k(\g)) \cong U(\g) / (e_\theta^{k+1}).
\]
For specific choices of $\g$ and $k$ this quotient can be computed explicitly. For example if $\g = \mathfrak{sl}_2$ and $k = 1$, one may compute
\begin{align*}
\zhu_0(V_1(\mathfrak{sl}_2)) \cong \mat_2(\C) \oplus \C.
\end{align*}
This algebra has the irreducible representations $\C$ and $\C^2$ which correspond to the degree zero pieces of the two integrable $\hat{\mathfrak{sl}_2}$-modules at level $1$, as one expects (see \cite{IDLA}).

In the same way $\zhu_p(V_k(\g)) = \zhu_p(V^k(\g)) / \zhu_p(I)$. Unfortunately, it seems to be much more difficult to compute $\zhu_p(I)$ explicitly when $p > 0$.

Let $k \in \Z_+$. The normalized positive energy $V_k(\g)$-modules $\{M^1, M^2, \ldots M^r\}$ are exactly the integrable $\hat{\g}$-modules. Since $V_k(\g)$ is rational, $\zhu_p(V_k(\g))$ is finite-dimensional and semisimple for all $p \in \Z_+$. The graded pieces $M^i_n$, where $1 \leq i \leq r$ and $0 \leq n \leq p$, form a complete list of irreducible $\zhu_p(V_k(\g))$-modules. None of the modules $M^i_n$ are duplicates, because $M^i_{n_1}$ and $M^i_{n_2}$ have different eigenvalues of $L_0$ for $n_1 \neq n_2$, and
\[
M^{i_1}_{n_1} \cong M^{i_2}_{n_2} \Rightarrow M^{i_1} \cong M^{i_2} \quad \text{up to degree shifts}.
\]

Therefore, if $d^i_n = \dim M^i_n$,
\begin{align*}
\zhu_p(V_k(\g)) \cong \bigoplus_{\substack{1 \leq i \leq r  \\ 0 \leq n \leq p}} \mat_{d^i_n} \C
\quad \text{and} \quad
\phi_p : \bigoplus_{\substack{1 \leq i \leq r  \\ 0 \leq n \leq p}} \mat_{d^i_n} \C \rightarrow \bigoplus_{\substack{1 \leq i \leq r  \\ 0 \leq n \leq p-1}} \mat_{d^i_n} \C
\end{align*}
is the natural projection. The inverse limit is now just the direct product
\begin{align*}
\varprojlim \zhu_p(V_k(\g)) \cong \prod_{\substack{1 \leq i \leq r  \\ n \in \Z_+}} \mat_{d^i_n} \C.
\end{align*}
These comments apply to any rational vertex algebra with a Virasoro element, and for which we know the graded dimensions of its irreducible modules.


\renewcommand{\theequation}{\Alph{section}.\arabic{equation}}

\renewcommand{\thethm}{\Alph{section}.\arabic{thm}}
\renewcommand{\thedefn}{\Alph{section}.\arabic{defn}}
\renewcommand{\theexmp}{\Alph{section}.\arabic{exmp}}
\renewcommand{\thermrk}{\Alph{section}.\arabic{rmrk}}

\appendix
\section{\texorpdfstring{Proof of equation (\ref{particular})}{Proof of equation (4.4)}} \label{appcalc}

In this section we continue in the set-up of section \ref{inducfunct}, i.e., $\tilde{M}$ is the $\lie V$-module defined there, $x \in N$, and $a, b \in V$ such that $[\eps_a] + [\eps_b] = 0$. To prove equation (\ref{particular}), it suffices to prove
\begin{align} \label{reducedpart}
\langle \psi, (a_{[-1]}b)^M_0 x - \sum_{j \in \Z_+} a^M_{p+\eps_a-j} b^M_{-p-\eps_a+j} x \rangle = 0.
\end{align}
Equivalently
\begin{align} \label{BIpp2}
\left< \psi, (a_{[-1]}b)^M_0 x \right>
= \sum_{k=0}^{p-\chi(a, b)} \left< \psi, a^M_{-k+\eps_a} b^M_{k-\eps_a} x \right> + \sum_{k=1}^p \left< \psi, a^M_{k+\eps_a} b^M_{-k-\eps_a} x \right>,
\end{align}
because $b^M_n x = 0$ when $n > p$.

Evaluating $\left<\psi, a^M_{-k+\eps_a} b^M_{k-\eps_a} x\right>$ for $k-\eps_a > 0$ involves a choice of $c \in V$ such that $a_{-k+\eps_a} b_{k-\eps_a} = c_0$. We cannot evaluate $\left< \psi, a^M_{k+\eps_a} b^M_{-k-\eps_a} x \right>$, where $k+\eps_a > 0$, in the same way, because the monomial is not ordered correctly. We must first use the commutation relations of $\lie V$ to rewrite $a^M_{k+\eps_a} b^M_{-k-\eps_a}$ in terms of our PBW basis.

The following lemma gives us a choice of $c$ to use henceforth.
\begin{lemma} \label{akbk}
Let $V$ be a $\G/\Z$-graded vertex algebra, and $M$ a $(\G, H)$-twisted $V$-module. Also let $x \in M_p$, let $a, b \in V$ such that $[\eps_a] + [\eps_b] = [0]$, and let $k$ be an integer such that $0 \leq k \leq p + \eps_a$. We have
\begin{align} \label{akbklem}
a^M_{-k+\eps_a} b^M_{k-\eps_a} x = \sum_{m=0}^{p-k} \binom{-p-1-k}{m} (a_{[-p-1-k-m]}b)^M_0 x.
\end{align}
\end{lemma}

\begin{proof}
We substitute equation (\ref{biggerp}) into the right hand side of equation (\ref{akbklem}), and remove terms of the form $a^M_n x$ for $n > p$, to obtain
\begin{align*}
\sum_{\substack{m, j \in \Z_+, \\ 0 \leq j+m \leq p-k}} (-1)^j \binom{-p-1-k}{m} \binom{-p-1-k-m}{j} a^M_{-k-m-j+\eps_a} b^M_{k+m+j-\eps_a} x.
\end{align*}
For an integer $\alpha$ such that $0 \leq \alpha \leq p-k$, the coefficient of $a^M_{-k-\alpha+\eps_a} b^M_{k+\alpha-\eps_a} x$ above is
\begin{align*}
\sum_{j=0}^{\alpha} (-1)^j \binom{-p-1-k}{\alpha-j} \binom{-p-1-k-\alpha+j}{j}
& = \sum_{j=0}^{\alpha} \binom{-p-1-k}{\alpha-j} \binom{p+k+\alpha}{j} \\
& = [\xi^{\alpha}]: (1+\xi)^{-p-1-k} (1+\xi)^{p+k+\alpha} \\
& = [\xi^{\alpha}]: (1+\xi)^{\alpha-1}.
\end{align*}
This is $1$ if $\alpha=0$ and $0$ if $\alpha>0$. This proves the lemma.
\end{proof}

We now have,
\begin{align} \label{akbkeqn}
\left<\psi, a^M_{-k+\eps_a} b^M_{k-\eps_a} x \right> = \sum_{m=0}^{p-k} \binom{-p-1-k}{m} \left<\psi, (a_{[-p-1-k-m]}b)^M_0 x \right>.
\end{align}
If $\eps_a = k = 0$ this equation follows from the fact that $N$ is a $\zhu_{p, \G}(V)$-module; for other values of $\eps_a$ and $k$, it follows from Lemma \ref{akbk} and Definition \ref{pairingdef}. For brevity let us omit $\left< \psi, \cdot \right>$ from now on.

Armed with equation (\ref{akbkeqn}), we rewrite the first summation on the right hand side of equation (\ref{BIpp2}) as
\begin{align} \label{squa}
\sum_{k=0}^{p-\chi(a, b)} \sum_{m=0}^{p-k} \binom{-p-1-k}{m} (a_{[-p-1-k-m]}b)^M_0 x
= \sum_{\alpha=0}^{p-\chi(a, b)} \binom{-p}{\alpha} (a_{[-p-1-\alpha]}b)^M_0 x.
\end{align}
To see this equality we first suppose $\chi(a, b) = 0$. For an integer $\al$ such that $0 \leq \al \leq p$, the coefficient of $(a_{[-p-1-\alpha]}b)^M_0$ is
\begin{align} \label{q}
\begin{split}
\sum_{k=0}^{\alpha} \binom{-p-1-k}{\alpha-k}
& = [\xi^{\alpha}]: (1+\xi)^{-p-1} \sum_{k=0}^{\infty} \xi^k (1+\xi)^{-k} \\
& = [\xi^{\alpha}]: (1+\xi)^{-p}
= \binom{-p}{\alpha}
\end{split}
\end{align}
(using the geometric series formula), so equation (\ref{squa}) is true. When $\chi(a, b) = 1$, equation (\ref{q}) is valid for $0 \leq \alpha \leq p-1$, but not $\alpha = p$. However, $N$ is a $\zhu_{p, \G}(V)$-module, so $(a_{[-2p-2+\chi(a, b)]}b)^M_0 x = (a_{[-2p-1]}b)^M_0 x = 0$ for $x \in N$. Therefore (\ref{squa}) holds if $\chi(a, b) = 1$ too.

Substituting equation (\ref{asndef}) into the right hand side of equation (\ref{squa}) reduces it to
\begin{align}
\sum_{\alpha=0}^{p-\chi(a, b)} & \binom{-p}{\alpha} \left[ \sum_{i \in \Z_+} \binom{\ga_a+p}{i} (a_{(-p-1-\alpha+i)}b)^M_0 \right] x \nonumber \\
= {} & \sum_{j \in \Z} \left[ \sum_{\alpha=0}^{p-\chi(a, b)} \binom{-p}{\alpha} \binom{\ga_a+p}{p+1+j+\alpha} \right] (a_{(j)}b)^M_0 x. \label{firsthalf}
\end{align}

Now we turn our attention to the second sum on the right hand side of equation (\ref{BIpp2}). First apply equation (\ref{conbrak}) to get
\begin{align} \label{1stofsecpart}
\sum_{k=1}^p a^M_{k+\eps_a} b^M_{-k-\eps_a} x
= p(a, b) \sum_{k=1}^p b^M_{-k-\eps_a} a^M_{k+\eps_a} x
+ \sum_{k=1}^p \sum_{j \in \Z_+} \binom{k + \ga_a - 1}{j} (a_{(j)}b)^M_0 x.
\end{align}

We put $l = k-\chi(a, b)$ and use $\eps_a = -\eps_b-\chi(a, b)$ to write the first summation on the right hand side of (\ref{1stofsecpart}) as
\begin{align*}
p(a, b) \sum_{l=1-\chi(a, b)}^{p-\chi(a, b)} b^M_{-l+\eps_b} a^M_{l-\eps_b} x.
\end{align*}
We now repeat calculation (\ref{squa})-(\ref{q}) to reduce this summation to
\begin{align}
p(a, b) \sum_{\alpha=0}^{p-\chi(a, b)} & \left[ \binom{-p}{\alpha} - \delta_{0, \chi(a, b)} \binom{-p-1}{\alpha} \right] (b_{[-p-1-\alpha]}a)^M_0 x \nonumber \\
= {} & p(a, b) \sum_{\alpha=0}^{p-\chi(a, b)} \binom{-p-1+\chi(a, b)}{\alpha-1+\chi(a, b)} (b_{[-p-1-\alpha]}a)^M_0 x. \label{inter2}
\end{align}

Equation (\ref{swapduple}), together with the fact that $N$ is a $\zhu_{p, \G}(V)$-module, implies
\begin{align*}
(b_{[-p-1-\alpha]}a)_0 x = p(a, b) (-1)^{p-\alpha} \sum_{j \in \Z} \binom{\alpha+\ga_a-1+\chi(a, b)}{j+p+1+\alpha} (a_{(j)}b)^M_0 x.
\end{align*}
Plugging this into equation (\ref{inter2}) yields
\begin{align} \label{secondhalf}
\sum_{j \in \Z} \left[ \sum_{\alpha=0}^{p-\chi(a, b)} (-1)^{p+\alpha} \binom{-p-1+\chi(a, b)}{\alpha-1+\chi(a, b)} \binom{\alpha+\ga_a-1+\chi(a, b)}{p+1+j+\alpha} \right] (a_{(j)}b)^M_0 x.
\end{align}

Combining (\ref{firsthalf}) with (\ref{secondhalf}) renders the right hand side of equation (\ref{BIpp2}) equal to
\begin{align*}
\sum_{k=1}^p \sum_{j \in \Z} & \binom{k + \ga_a - 1}{j} (a_{(j)}b)^M_0 x
+ \sum_{j \in \Z} \sum_{\alpha=0}^{p-\chi(a, b)} \left[ \binom{-p}{\alpha} \binom{\ga_a+p}{p+1+j+\alpha} \right. \\
& + \left. (-1)^{p+\alpha} \binom{-p-1+\chi(a, b)}{\alpha-1+\chi(a, b)} \binom{\alpha+\ga_a-1+\chi(a, b)}{p+1+j+\alpha} \right] (a_{(j)}b)^M_0 x.
\end{align*}
The left hand side of (\ref{BIpp2}) is
\[
\sum_{j \in \Z} \binom{\ga_a+p}{j+1} (a_{(j)}b)^M_0 x.
\]

Before proving the equality of these two expressions in general, we dispense with the special case $p = 0$, $\chi(a, b) = 1$. In this case the left hand side vanishes because $N$ is a $\zhu_{0, \G}(V)$-module and the summation of the right hand side is over an empty range. We exclude this case from now on.

In the other cases we prove equality for the coefficient of $(a_{(j)}b)^M_0$ and thus for the whole expression.

We claim that
\begin{align} \label{shortlem}
\sum_{k=1}^p \binom{k+\ga-1}{j} = \binom{\ga+p}{j+1} - \binom{\ga}{j+1}.
\end{align}
Indeed the left hand side of (\ref{shortlem}) is
\begin{align*}
[\xi^j]: (1+\xi)^{\ga} \sum_{k=0}^{p-1} (1+\xi)^{k}
&= [\xi^j]: (1+\xi)^{\ga} \frac{1 - (1+\xi)^p}{1 - (1+\xi)} \\
&= [\xi^j]: \frac{(1+\xi)^{\ga} - (1+\xi)^{p+\ga}}{-\xi},
\end{align*}
which equals the right hand side.

It therefore remains to prove that
\begin{align}
\begin{split}
\binom{\ga_a}{j+1}
= {} & \sum_{k=0}^{p-\chi(a, b)} \left[ \binom{-p}{k} \binom{\ga+p}{p+j+k+1} \right. \\
& + \left. (-1)^{p+k} \binom{-p-1+\chi(a, b)}{k-1+\chi(a, b)} \binom{k+\ga-1+\chi(a, b)}{p+1+j+k} \right]
\end{split} \label{Sequals}
\end{align}
for $p, j \in \Z_+$, $\ga$ arbitrary and $(p, \chi) \neq (0, 1)$. Equation (\ref{Sequals}) is a special case of Lemma \ref{bigcomblem} below (with $(\ga, n, X, Y) = (\ga_a, j+1, p, p+1-\chi(a, b))$).

\begin{lemma} \label{bigcomblem}
For $\ga \in \R$, $n \in \Z$, and $X, Y \in \Z$ with $X \geq 0$, $Y \geq 1$, let
\begin{align*}
H_{\ga, n}(X, Y) &= \sum_{k=0}^{Y-1} \binom{-X}{k} \binom{\ga+X}{n+X+k} \\
\text{and} \quad
D_{\ga, n}(X, Y) &= \sum_{k=0}^{X-1} (-1)^{Y+k} \binom{-Y}{k} \binom{\ga+k}{n+Y+k}.
\end{align*}
Then
\begin{align*}
H_{\ga, n}(X, Y) + D_{\ga, n}(X, Y) = \binom{\ga}{n}.
\end{align*}
\end{lemma}

\begin{proof}
Fix $\ga$ and $n$, for a pair $(X, Y)$ call the claim $P(X, Y)$. For $Y \geq 1$, $P(0, Y)$ is clear. Suppose $P(X, Y)$ holds for some $(X, Y)$, we will deduce $P(X+1, Y)$, the claim follows by induction.

Using $\binom{\alpha+1}{s+1} = \binom{\alpha}{s+1} + \binom{\alpha}{s}$ we have
\begin{align*}
H_{\ga, n}(X+1, Y)
= {} & \sum_{k=0}^{Y-1} \binom{-X-1}{k} \binom{\ga+X}{n+X+1+k} + \sum_{k=0}^{Y-1} \binom{-X-1}{k} \binom{\ga+X}{n+X+k} \\
= {} & \binom{-X-1}{Y-1} \binom{\ga+X}{n+X+Y+1} \\
& + \sum_{k=0}^{Y-1} \left[ \binom{-X-1}{k-1} + \binom{-X-1}{k} \right] \binom{\ga+X}{n+X+k} \\
= {} & (-1)^{X+Y+1} \binom{-Y}{X} \binom{\ga+X}{n+X+Y} + H_{\ga, n}(X, Y).
\end{align*}
Therefore
\begin{align*}
H_{\ga, n}(X, Y) + D_{\ga, n}(X, Y)
= {} & H_{\ga, n}(X+1, Y) + D_{\ga, n}(X, Y) \\
& + (-1)^{X+Y} \binom{-Y}{X} \binom{\ga+X}{n+X+Y} \\
= {} & H_{\ga, n}(X+1, Y) + D_{\ga, n}(X+1, Y).
\end{align*}
\end{proof}


\section{Extension to non-integer level} \label{pnotinZ}

In several places we promised to describe the construction of $\zhu_{p, \G, \hbar}(V)$ for $p \notin \Z_+$. In particular we require them for Theorem \ref{rati}. We have seen that in an indecomposable $(\G, H)$-twisted positive energy $V$-module, the degrees of the graded pieces lie in a single coset of $\overline{\G}$. Thus it makes sense to define Zhu algebras for all $p \in \overline{\G} \cap \R_+$. However, it is more convenient for us to adopt a system in which $\zhu_{p, \G}(V)$ is defined for arbitrary $p \in \R_+$, and such that if $q < p$ and $[q, p] \cap \overline{\G}$ is empty then the level $p$ and level $q$ Zhu algebras are isomorphic. Such a system allows us to state that `for every twisted positive energy $V$-module $M$, $M_p$ is a $\zhu_{p, \G}(V)$-module' without qualification. In this appendix we will use a slightly changed notation, letting $P$ denote the level of the Zhu algebra and reserving $p$ to stand for the integer part $\fp$.

\subsection{New notations} \label{newnot}

Let $a, b \in V$ such that $[\eps_a] + [\eps_b] = [0]$, and let $x \in M_P$. We repeat the calculation of section \ref{motivation}, with $p+1+\eps_a$ replaced by $P_a$, which we define to be the smallest element of $[\eps_a]$ that is strictly greater than $P$.

Equation (\ref{biggerp}) becomes
\begin{align}
\sum_{j \in \Z_+} \binom{P_a+\D_a-1}{j}(a_{(n+j)}b)^M_0 x
= \sum_{j \in \Z_+} (-1)^j \binom{n}{j} a^M_{P_a+n-j} b^M_{-P_a-n+j} x. \label{biggerpprime}
\end{align}
The right hand side equals zero when $n < -P-P_a$, so we define $N_a = N([\eps_a], P)$ to be the largest integer strictly less than $-P-P_a$.

The modified state-field correspondence should now be
\begin{align*}
Z(a, z) = (1 + \hbar z)^{\xi_a} Y(a, z) = \sum_{n \in \Z} a_{[n]} z^{-n-1},
\end{align*}
where $\xi_a = P_a + \D_a - 1$ is the largest element of $[\ga_a]$ that does not exceed $P + \D_a$. We have
\begin{align*}
a_{[n]}b = \sum_{j \in \Z_+} \binom{\xi_a}{j} \hbar^j a_{(n+j)}b.
\end{align*}

In the case that $\eps_a = \eps_b = 0$, we have simpler expressions for these quantities:
\begin{align} \label{epsez}
P_a = \fp + 1, \quad \xi_a = \fp + \D_a \quad \text{and} \quad N_a = -2\fp-2.
\end{align}

Writing $p = \fp$ makes many formulas from the rest of the paper remain true verbatim. For example, for $a, b \in V_\G$ we define $a *_P b$ to be
\begin{align}
a *_P b = \sum_{m=0}^p \binom{-p-1}{m} \hbar^{-p-m} a_{[-p-1-m]}b,
\end{align}
when we set $\hbar = 1$ we find $(a *_P b)^M_0 = a^M_0 b^M_0$ on $M_P$.

The subspace $J_{P, \G, \hbar}$ is now defined to be
\begin{align*}
J_{P, \G, \hbar} = \C[\hbar, \hbar^{-1}] \{(T + \hbar H)a | [\eps_a] = 0\} + \C[\hbar, \hbar^{-1}] \{a_{[N_a]}b | [\eps_a] + [\eps_b] = [0]\}.
\end{align*}

\begin{rmrk}
Although the product $*_P$ depends only on $p = \fp$, the algebra $\zhu_{P, \G}(V)$ is not necessarily isomorphic to $\zhu_{p, \G}(V)$ because $J_{P, \G}$ and $J_{p, \G}$ may differ. For example let $V$ be a graded vertex algebra for which $\overline{\G} = \frac{1}{2}\Z$.
\begin{itemize}
\item $P = 0$: If $\eps_a = 0$ then $P_a = 1$, and so $N_a = -2$. If $\eps_a = -\frac{1}{2}$ then $P_a = \frac{1}{2}$, and so $N_a = -1$.

\item $P = \frac{1}{2}$: If $\eps_a = 0$ then $P_a = 1$ and so $N_a = -2$. If $\eps_a = -\frac{1}{2}$ then $P_a = \frac{3}{2}$, and so $N_a = -3$.
\end{itemize}
Thus $J_{1/2, \G}$ is strictly smaller than $J_{0, \G}$.
\end{rmrk}

The inclusion
\begin{align}
((T + \hbar H)V_\G) *_P V_\G \subseteq J_{P, \G, \hbar}
\end{align}
is proved just as before.

\subsection{The Borcherds identity}

The derivation of the modified Borcherds identity follows the same course as the $P \in \Z_+$ case. It is made more complicated by the piecewise nature of the functions $P_a$, $N_a$, and $\xi_a$. Just as we introduced a function $\chi(a, b)$ at the end of section \ref{prelim} that related $\eps_{a_{(n)}b}$ to $\eps_a$ and $\eps_b$, so now we introduce $\sigma(a, b)$ relating $\xi_{a_{(n)}b}$ to $\xi_a$ and $\xi_b$.
\begin{defn}
\begin{align}
\sigma(a, b) = \xi_{a_{(-1)}b} + p - \xi_a - \xi_b.
\end{align}
\end{defn}

We have
\begin{align}
\xi_{a_{(n)}b} = \xi_a + \xi_b - p - n - 1 + \sigma(a, b).
\end{align}

In Lemma \ref{brakderivlemma} we used the fact that $\chi(a, b) = 0$ whenever $\eps_a = 0$. For the general case we need the same fact for $\sigma$.
\begin{lemma} \label{ifaisgood}
If $\eps_a = 0$, then $\sigma(a, b) = 0$ for all $b \in V$.
\end{lemma}

\begin{proof}
$\xi_{a_{(-1)}b}$ is the largest element of $[\ga_{a_{(-1)}b}] = [\ga_a] + [\ga_b] = [\D_a] + [\ga_b]$ that does not exceed $P + \D_{a_{(-1)}b} = P + \D_a + \D_b$. But this is just $\D_a$ plus the largest element of $[\ga_b]$ that does not exceed $P + \D_b$, viz. $\D_a + \xi_b$. Now, $\eps_a = 0$ implies $\xi_a = p + \D_a$, so we have $\xi_{a_{(-1)}b} = \xi_a + \xi_b - p$. Therefore $\sigma(a, b) = 0$.
\end{proof}

\begin{lemma} \label{coolindex}
If $[\eps_a] + [\eps_b] = [0]$, then
\begin{align}
N_a = -2p - 2 + \sigma(a, b).
\end{align}
\end{lemma}

\begin{proof}
Recall that $P_a$ is defined to be the smallest element of $[\eps_a]$ strictly greater than $P$. As a function of $P$, we have $P_a(P + 1) = P_a(P) + 1$. On the interval $P \in [0, 1)$ we have
\begin{align*}
P_a = \left\{ \begin{array}{ll}
\eps_a + 1 & \text{if $0 \leq P < 1 + \eps_a$,} \\
\eps_a + 2 & \text{if $1 + \eps_a \leq P < 1$.} \\
\end{array} \right.
\end{align*}
Consequently $N(P) = N_a = -\left\lfloor P + P_a \right\rfloor - 1$ satisfies $N(P+1) = N(P)-2$, and on the interval $[0, 1)$ we have
\begin{align*}
N = \left\{ \begin{array}{ll}
-1 & \text{if $0 \leq P < E$,} \\
-2 & \text{if $E \leq P < E'$,} \\
-3 & \text{if $E \leq P < 1$,} \\
\end{array} \right.
\end{align*}
where $E = \min\{1 + \eps_a, -\eps_a\}$, and $E' = \max\{1 + \eps_a, -\eps_a\}$.

Since $[\eps_a] + [\eps_b] = [0]$, we have $\eps_{a_{(n)}b} = 0$, hence $\xi_{a_{(n)}b} = p + \ga_{a_{(n)}b}$ (cf. equation (\ref{epsez})). Therefore
\begin{align*}
\sigma(a, b)
& = \xi_{a_{(-1)}b} + p - \xi_a - \xi_b \\
& = p + \ga_{a_{(-1)}b} + p - (p + \ga_a) - (p + \ga_b) + (p + \ga_a - \xi_a) + (p + \ga_b - \xi_b) \\
& = (p + \ga_a - \xi_a) + (p + \ga_b - \xi_b) + \chi(a, b).
\end{align*}

As a function of $P$, $p + \ga_a - \xi_a$ is periodic with period $1$. On the interval $P \in [0, 1)$, $p(P) = 0$, and
\begin{align*}
p + \ga_a - \xi_a = \left\{ \begin{array}{ll}
0 & \text{if $0 \leq P < 1 + \eps_a$,} \\
-1 & \text{if $1 + \eps_a \leq P < 1$.} \\
\end{array} \right.
\end{align*}

There are now two cases. First suppose $\eps_a = \eps_b = 0$, then $p + \ga_a - \xi_a = p + \ga_b - \xi_b = 0$ and $\chi(a, b) = 0$, hence $\sigma(a, b) = 0$. Meanwhile $N_a = -2p-2$ so we are done.

Now suppose that $\eps_a \neq 0$, hence $\eps_a + \eps_b = -1$ and $\chi(a, b) = 1$. Suppose without loss of generality that $\eps_a \geq \eps_b$. Then
\begin{align*}
(p + \ga_a - \xi_a) + (p + \ga_b - \xi_b) = \left\{ \begin{array}{ll}
0 & \text{if $0 \leq P < 1 + \eps_b$,} \\
-1 & \text{if $1 + \eps_b \leq P < 1 + \eps_a$,} \\
-2 & \text{if $1 + \eps_a \leq P < 1$.} \\
\end{array} \right.
\end{align*}
Thus $N_a = -2p - 2 + \sigma(a, b)$ for $P \in [0, 1)$, and hence for all $P$.
\end{proof}

We can rewrite the modified $n^{\text{th}}$-product identity, the modified Borcherds identity, and equation (\ref{expbor}) for the general case. In each formula we simply replace $\chi$ with $\sigma$. Because of Lemma \ref{coolindex}, the statements made at the end of section \ref{borsec}, with $\chi$ replaced by $\sigma$, are true in the general case. We conclude that $J_{P, \G, \hbar}$ is a right ideal of $(V_\G, *_P)$.

\subsection{Skew-symmetry}

The calculations of section \ref{sksmsec} work with $\xi_a$ in place of $p+\ga_a$. Equation (\ref{swapduple}) now has $(\xi_b-\D_b-\D_a)$ replacing $(p-\ga_a+\eps_a+\eps_b)$ in the top of the binomial coefficient. The other important formula of the section is (\ref{swapexpl}); it remains unchanged. The skew-symmetry formula holds as before, with $[\cdot, \cdot]_\hbar$ defined as before. The proof that $J_{P, \G, \hbar}$ is a left ideal goes as before. Thus $\zhu_{P, \G, \hbar}(V)$ is well-defined.

\subsection{\texorpdfstring{Remaining properties of $\zhu_{P, \G, \hbar}(V)$}{Remaining properties of the Zhu algebra}}

The sections on associativity and unitality carry over verbatim.

Let $P \geq Q \geq 0$. We claim that the identity map on $V_\G$ induces a surjective algebra homomorphism
\[
\phi_{P, Q} : \zhu_{P, \G, \hbar}(V) \twoheadrightarrow \zhu_{Q, \G, \hbar}(V).
\]
To prove this, we first assume that $P - Q \leq 1$.

Suppose $Q < \ga-\D_a \leq P$ for some $\ga \in [\ga_a]$. Then $\xi_{a, P} = \xi_{a, Q} + 1$ and
\begin{align*}
a_{[n, P]}b = a_{[n, Q]}b + \hbar a_{[n+1, Q]}b.
\end{align*}
In Lemma \ref{coolindex} we wrote down a formula for $N_{a, P}$ explicitly as a function of $P$. It has a jump discontinuity of size $1$ at each $\eps \in [\eps_a]$. At $P = \eps$, it takes the lower value. Hence $N_{a, P} \leq N_{a, Q} - 1$. This implies that $J_{P, \G, \hbar} \subseteq J_{Q, \G, \hbar}$.

Now suppose that there is no such $\ga \in [\ga_a]$. Then $\xi_{a, P} = \xi_{a, Q}$, hence $a_{[n, P]}b = a_{[n, Q]}b$. Because $N_{a, P} \leq N_{a, Q}$, we have $J_{P, \G, \hbar} \subseteq J_{Q, \G, \hbar}$.

The main calculation of section \ref{maps} shows that $a_{[n, P]}b \equiv a_{[n, Q]}b \pmod{J_{Q, \G, \hbar}}$. So $\phi_{P, Q}$ makes sense for $P - Q \leq 1$. It is defined for more widely separated $P$ and $Q$ by composing the maps defined above. This definition is sound because all the maps are induced by the identity on $V_\G$.

\subsection{Representation theory}

Section \ref{restrictfunct} carries over, essentially unmodified, to the general case. The same goes for section \ref{inducfunct}. Equation (\ref{particular}), which is proved in appendix \ref{appcalc}, must be replaced by
\begin{align}
\left< \psi, BI(a, b; P_a, -P_a; -1) x \right> = 0, \label{particularprime}
\end{align}
which is proved in the next section.

The functorial statements of section \ref{funcprop} carry over to the general case as well. Now we define a $P$-founded $\zhu_{P, \G}(V)$-module to be one that does not factor to a $\zhu_{Q, \G}(V)$-modules for \emph{each} $Q < P$. The proof of Theorem \ref{fullequiv}, which uses Lemma \ref{akbk}, is easily extended to the general case, using Lemma \ref{akbkprime} instead. The remainder of section \ref{repnthry} is the same for $P \notin \Z_+$ as it was for $P \in \Z_+$.

\subsection{\texorpdfstring{Proof of equation (\ref{particularprime})}{Proof of equation (B.8)}}

Lemma \ref{akbk} generalizes to the following.
\begin{lemma} \label{akbkprime}
Let $V$ be a $\G/\Z$-graded vertex algebra, and $M$ a $(\G, H)$-twisted $V$-module, let $x \in M_P$, let $a, b \in V$ be such that $[\eps_a] + [\eps_b] = [0]$. Also let $P_a$ and $N_a$ be as in section \ref{newnot}. Define
\begin{align}
R_a = R([\eps_a], P) = P_a - \eps_a = \left\lfloor P - \eps_a \right\rfloor + 1,
\end{align}
and let $k$ be an integer such that $0 \leq k \leq P_a-1$. We have
\begin{align}
a^M_{-k+\eps_a} b^M_{k-\eps_a} x = \sum_{m \in \Z_+} \binom{-R_a-k}{m} (a_{[-R_a-k-m]}b)^M_0 x. \label{akbklemprime}
\end{align}
\end{lemma}

Now equation (\ref{akbkeqn}) is replaced by
\begin{align} \label{akbkeqnprime}
\left<\psi, a^M_{-k+\eps_a} b^M_{k-\eps_a} x \right> = \sum_{m \in \Z_+} \binom{-R_a-k}{m} \left<\psi, (a_{[-R_a-k-m]}b)^M_0 x \right>.
\end{align}
We may use this equation to reduce (\ref{particularprime}) to a combinatorial identity as before. This time we arrive at
\begin{align} \label{2ndtolast}
\begin{split}
\binom{\ga_a}{j+1}
= {} & \sum_{k=0}^{R_b-\chi(a, b)-1} \binom{-R_a+1}{k} \binom{\xi_a}{R_a+j+k} \\
& + \sum_{k=0}^{R_a-\chi(a, b)-1} (-1)^{R_b+k+1} \binom{-R_b+\chi(a, b)}{k-1+\chi(a, b)} \binom{k+\ga_a-1+\chi(a, b)}{R_b+j+k}.
\end{split}
\end{align}

If $\chi(a, b) = 0$, then $P_a = P_b = p+1 = R_a = R_b$ and $\xi_a = p + \ga_a$. Substituting these values reduces (\ref{2ndtolast}) to (\ref{Sequals}), which we have already proved. If $\chi(a, b) = 1$, we have
\begin{align}
\xi_a = P_a + \D_a - 1 = P_a + \ga_a - \eps_a - 1 = R_a + \ga_a - 1,
\end{align}
and the same equations for $b$. Substituting these values reduces (\ref{2ndtolast}) to Lemma \ref{bigcomblem} with $(\ga, n, X, Y) = (\ga_a, j+1, R_a-1, R_b-1)$.


\clearpage

\begin{thebibliography}{99}

\bibitem{DK}
A. De Sole, V. G. Kac,
\newblock ``Finite vs affine $W$-algebras,''
\newblock Jpn. J. Math. \textbf{1}, 137-261 (2006).

\bibitem{DLMorbifold}
C. Dong, H. Li, G. Mason,
\newblock ``Modular-invariance of trace functions in orbifold theory and generalized Moonshine,''
\newblock Comm. Math. Phys. \textbf{214}, 1-56 (2000).

\bibitem{DLM2}
C. Dong, H. Li, G. Mason,
\newblock ``Twisted representations of vertex operator algebras,''
\newblock Math. Ann. \textbf{310}, 571-600 (1998).

\bibitem{DLM3}
C. Dong, H. Li, G. Mason,
\newblock ``Twisted representations of vertex operator algebras and associative algebras,''
\newblock Internat. Math. Res. Notices \textbf{8}, 389-397 (1998).

\bibitem{DLM}
C. Dong, H. Li, G. Mason,
\newblock ``Vertex operator algebras and associative algebras,''
\newblock  J. Algebra \textbf{206}, 67-96 (1998).

\bibitem{FZ}
I. Frenkel, Y. Zhu,
\newblock ``Vertex operator algebras associated to representations of affine and Virasoro algebras,''
\newblock Duke Math. J. \textbf{66}, 123-168 (1992).

\bibitem{IDLA}
V. G. Kac,
\newblock \emph{Infinite dimensional Lie algebras},
\newblock 3rd ed. (Cambridge University Press, Cambridge, 1990).

\bibitem{Kac}
V. G. Kac,
\newblock \emph{Vertex Algebras for Beginners},
\newblock University Lecture Series No. 10, 2nd ed. (American Mathematical
Society, Providence, RI, 1998).

\bibitem{Li}
H. Li,
\newblock ``Representation theory and tensor product theory for vertex operator algebras,''
\newblock Ph.D. thesis, Rutgers University, 1994.

\bibitem{Zhu}
Y. Zhu,
\newblock ``Modular invariance of characters of vertex operator algebras,''
\newblock J. Amer. Math. Soc. \textbf{9}, 237-302 (1996).


\end{thebibliography}
\end{document}